\documentclass{icmart}

\usepackage{enumitem}
\usepackage{hyperref}
\usepackage[initials]{amsrefs}
\usepackage[all]{xy}
\SelectTips{cm}{}

\contact[uri.bader@gmail.com]{Uri Bader}
\contact[furman@uic.edu]{Alex Furman}

\newcommand{\bbN}{{\mathbb N}}
\newcommand{\bbQ}{{\mathbb Q}}
\newcommand{\bbE}{{\mathbb E}}
\newcommand{\bbR}{{\mathbb R}}
\newcommand{\bbZ}{{\mathbb Z}}
\newcommand{\bbC}{{\mathbb C}}

\newcommand{\calF}{\mathcal{F}}
\newcommand{\calN}{\mathcal{N}}
\newcommand{\calM}{\mathcal{M}}

\newcommand{\calZ}{\mathcal{Z}}

\newcommand{\pr}{\operatorname{pr}}

\newcommand{\supp}{\operatorname{supp}}

\newcommand{\Prob}{\operatorname{Prob}}

\newcommand{\Is}{\operatorname{Iso}}

\newcommand{\Homeo}{\operatorname{Homeo}}

\newcommand{\Commen}{\operatorname{Commen}}
\newcommand{\SL}{\operatorname{SL}}
\newcommand{\PSL}{\operatorname{PSL}}

\newcommand{\Aut}{\operatorname{Aut}}

\newcommand{\Map}{\operatorname{Map}}
\newcommand{\bnd}{\textbf{bnd}}

\newcommand{\overto}[1]{{\buildrel{#1}\over\longrightarrow}}

\newcommand{\acts}{\curvearrowright}

\newcommand{\gW}[2]{\operatorname{W}_{#1,#2}}
\newcommand{\cW}[1]{\operatorname{Weyl}_{#1}}

\newcommand{\wflip}{w_{\rm flip}}
\newcommand{\wlong}{\operatorname{w}_{\rm long}}

\newcommand{\wt}[1]{\widetilde{#1}}
\newcommand{\ol}[1]{\overline{#1}}
\newcommand{\ec}{/\!\!/}

\newtheorem{theorem}{Theorem}[section]
\newtheorem{corollary}[theorem]{Corollary}
\newtheorem{lemma}[theorem]{Lemma}

\newtheorem{proposition}[theorem]{Proposition}

\theoremstyle{definition}
\newtheorem{definition}[theorem]{Definition}
\newtheorem{example}[theorem]{Example}
\newtheorem{observation}[theorem]{Observation}
\newtheorem{remark}[theorem]{Remark}
\newtheorem{remarks}[theorem]{Remarks}

\title[Boundaries, rigidity of representations, and Lyapunov exponents]{Boundaries, rigidity of representations, and Lyapunov exponents}

\author[Uri Bader, Alex Furman
]{Uri Bader\thanks{U.B was supported in part by the ERC grant 306706.}, Alex Furman\thanks{A.F. was supported in part by the NSF grants DMS 1207803.}}

\begin{document}

\begin{abstract}
In this paper we discuss some connections between measurable dynamics and 
rigidity aspects of group representations and group actions.
A new ergodic feature of familiar group boundaries is introduced, and is used to obtain
rigidity results for group representations and to prove simplicity of Lyapunov exponents
for some dynamical systems.
\end{abstract}

\begin{classification}
Primary 37A; Secondary 22E.
\end{classification}

\begin{keywords}
Boundary theory, isometric ergodicity, characteristic maps, superrigidity, Lyapunov exponents.
\end{keywords}

\maketitle


\section{Introduction}

%

Boundary theory is a broad term referring to constructions of auxiliary spaces that are used to analyze
asymptotic properties of spaces and groups, to study representations and group actions, and for other applications.
The topics discussed in this paper revolve around rigidity phenomena, inspired by Margulis' superrigidity,
and are then connected to the problem of simplicity of the Lyapunov exponents in classical dynamics.
Much of the work on which this paper is based is yet unpublished.
So rather than aiming at outmost generality,
we chose to illustrate the main ideas by presenting key results and to include sketches of their proofs.
Results about representations have natural cocycle versions;
and while we focus here on real Lie groups much of the work can be extended
to algebraic groups over more general fields.

\subsection*{Notations} 

The abbreviation for locally compact secondly countable group is \emph{lcsc}.
We shall use symbols $G$, $H$, $S$, and even $\Gamma$ to denote lcsc groups;
with $\Gamma$ being often discrete countable group, and $G$, $H$ mostly used for
real Lie groups, or (real points of) algebraic groups over $\bbR$.

By an action $\Gamma\acts X$ of a group $\Gamma$ on a set $X$ we mean a map
$\Gamma\times X\to X$, $(g,x)\mapsto g.x$,
so that $e.x=x$ and $gh.x=g.(h.x)$ for every $g,h\in \Gamma$, $x\in X$.
If $\Gamma$ is a lcsc group, a \emph{Borel $\Gamma$-space} $X$ is a standard Borel space $X$
with a $\Gamma $-action for which $\Gamma \times X\to X$ is a Borel map.
A \emph{Lebesgue $\Gamma$-space} is a Borel $\Gamma $-space $X$ with a Borel probability measure
$m$ on $X$ that is quasi-invariant under every $g\in \Gamma $, i.e. $g_*m\sim m$ for all $g\in \Gamma$.
A Lebesgue $\Gamma$-space $(X,m)$ is \emph{ergodic} if the only measurable $E\subset X$
with $m(g^{-1}E\triangle E)=0$ for every $g\in \Gamma$ satisfies $m(E)=0$ or $m(E)=1$.
The notion of a Lebesgue $\Gamma$-space depends only on the measure class $[m]$ of $m$;
a reference to $m$ will often be omitted from the notion for Lebesgue $G$-spaces.
If a Lebesgue $\Gamma$-space $X$ has a probability measure $m$ that is actually $\Gamma$-invariant,
i.e. $m(g^{-1}E)=m(E)$
for every measurable $E\subset X$ and $g\in\Gamma$,
we will say that the action $\Gamma \acts (X,m)$ is \emph{probability measure preserving} (p.m.p.).
If $X$ is a topological space (in particular a compact Hausdorff space) 
and the action map $\Gamma \times X\to X$ is continuous, we say
that $X$ is a \emph{topological $\Gamma$-space} (a \emph{compact $\Gamma$-space}).
A topological $\Gamma$-space $X$ is \emph{proper} if for every compact subset $Q\subset X$
the set $\{ g\in\Gamma \mid gQ\cap Q\ne \emptyset\}$ is precompact in $\Gamma$.
We shall use the notation $X^\Gamma$ for the set of $\Gamma$-fixed points in $X$.

Let $\Gamma $ be a lcsc group, $X$ a Lebesgue $\Gamma $-space, and $V$ a Borel $\Gamma $-space.
A measurable $\Gamma$-\emph{equivariant function} is a Borel function $f:X\to V$ such that for every $g\in \Gamma$, $f(g.x)=g.f(x)$
for a.e. $x\in X$.
We denote by $\Map_\Gamma (X,V)$
the space of equivalence classes of $\Gamma$-equivariant functions, where functions that agree $m$-a.e. are identified.
We shall use the term $\Gamma$-map to describe such a class $\phi\in \Map_\Gamma (X,V)$.
By a $\Gamma$-map $p:X\to Y$ between Lebesgue $\Gamma$-spaces we mean a $\Gamma$-map
for which $p[\mu] =[\nu]$, where $[\mu]$ denotes the $\Gamma$-invariant
measure class on $X$ and $[\nu]$ the one on $Y$.

\section{Boundaries} 

In this section we introduce a version of the concept of a boundary,
or rather boundary pair (Definition~\ref{D:boundary-pairs}),
and discuss the basic properties of this notion. A more detailed discussion will appear in \cite{BF-bnd}.
In our context a $\Gamma$-boundary is a Lebesgue $\Gamma$-space,
and as such it may have many realizations on topological $\Gamma$-spaces.
Furthermore, even as a Lebesgue space a $\Gamma$-boundary may not be unique.

\subsection*{Isometric ergodicity} 

A Lebesgue $\Gamma$-space $(X,m)$ is \emph{isometrically ergodic} if for any isometric action $\Gamma\to \Is(M,d)$ on a separable metric space $(M,d)$,
any $\Gamma$-map $f:X\to M$ is essentially constant; in which case its essential value is a $\Gamma$-fixed point.
In short,
\[
	\Map_\Gamma(X,M)=\Map(X,M^\Gamma).
\]
Isometric ergodicity implies ergodicity, by taking the two point space $M=\{0,1\}$ with the trivial $\Gamma$-action.
Isometric ergodicity is a natural strengthening of \emph{ergodicity with unitary coefficients},
introduced by Burger and Monod \cite{Burger+Monod}, where one considers only Hilbert spaces $M$ with unitary $\Gamma$-representations.
For p.m.p. actions $\Gamma\acts (X,m)$ both notions are equivalent to weak-mixing (cf. \cite{GW}).
However, here we shall be mostly interested in Lebesgue $\Gamma$-spaces
that have no invariant probability measure in the relevant measure class.

Next we introduce a \emph{relative} notion of isometric ergodicity, or equivalently, isometric ergodicity of
$\Gamma$-maps between Lebesgue $\Gamma$-spaces.
We first defined a relative notion of a metric space. Given a Borel map $q:\calM\to V$ between standard Borel spaces, a metric on $q$ is a Borel function
$d:\calM\times_V\calM\to[0,\infty]$ whose restriction $d_v$ to each fiber $M_v=q^{-1}(\{v\})$ is a separable metric.
A \emph{fiber-wise isometric $\Gamma$-action} on such $\calM$ consists of $q$-compatible actions $\Gamma\acts \calM$, $\Gamma\acts V$,
so that the maps between the fibers $g:M_v\to M_{g.v}$ are isometries, i.e.
\[
	d_{g.v}(g.x,g.y)=d_v(x,y)\qquad (x,y\in M_v,\ v\in V,\ g\in\Gamma).
\]
\begin{definition}
	A map $p:A\to B$ between Lebesgue $\Gamma$-spaces is \emph{relatively isometrically ergodic}
	if for every fiber-wise isometric $\Gamma$-action on $\calM\to V$ as above,
	and for any $q$-compatible $\Gamma$-maps $f:A\to \calM$, $f_0:B\to V$,
	there is a compatible $\Gamma$-map $f_1:B\to\calM$ making the following diagram commutative:
	\[
		\xymatrix{
		A \ar[d]_p \ar[r]^f & \calM \ar[d]^q\\
		B \ar[r]^{f_0} \ar@{.>}[ur]^{f_1} & V.
		}
	\]
\end{definition}
Note that isometric ergodicity relative to the trivial action on a point is just the (absolute) isometric ergodicity.
Let us list without proofs some basic properties of the notion of relatively isometrically ergodic maps.
Some of them are reminiscent of properties of relatively weakly mixing extensions in the context of p.m.p. actions.
In fact, for p.m.p. actions, or more generally relatively p.m.p. maps between Lebesgue $\Gamma$-spaces,
relative isometric ergodicity is equivalent to relative weak mixing.
However this remark will play no role in the sequel.
\begin{proposition}\label{P:riser-props}
	\begin{itemize}
		\item[{\rm (i)}] The property of relative isometric ergodicity is closed under composition of $\Gamma$-maps.
		\item[{\rm (ii)}]
		If $A\to B\to C$ are $\Gamma$-maps, and $A\to C$ is relatively isometrically ergodic,
		then so is $B\to C$, but $A\to B$ need not be relatively isometrically ergodic.
		\item[{\rm (iii)}]
		For Lebesgue $\Gamma$-spaces $A$ and $B$, if the projection $A\times B\to B$ 
		is relatively isometrically ergodic then $A$ is (absolutely) isometrically ergodic. 
		This is  an "if and only if" in case $B$ is a p.m.p action, but not in general.
		\item[{\rm (iv)}]
		If $\Gamma$ is a lattice in a lcsc group $G$, and $p:A\to B$ is a relatively isometrically ergodic $G$-maps between $G$-spaces,
		then $p:A\to B$ remains relatively isometrically ergodic if viewed as a $\Gamma$-map between $\Gamma$-spaces.
	\end{itemize}
\end{proposition}

\subsection*{Boundary pairs} 

Recall the notion of an amenable action, or amenable Lebesgue $\Gamma$-space in the sense of Zimmer \cite{Zimmer-amen}.
We shall use the fact that if $B$ is an amenable Lebesgue $\Gamma$-space, then 
given a metrizable compact $\Gamma$-space $M$, the set $\Map_\Gamma(B, \Prob(M))$ is non-empty, i.e. there exist $\Gamma$-map
$\phi:B\to \Prob(M)$.
More generally, given an affine $\Gamma$-action on a convex weak-* compact set $Q\subset E^*$, 
where $E$ is a separable Banach space, there exists a $\Gamma$-map $\phi\in \Map_\Gamma(B, Q)$.

\begin{definition}\label{D:boundary-pairs}\hfill{}\\
Let $\Gamma$ be a lcsc group. A pair $(B_-, B_+)$ of Lebesgue $\Gamma$-spaces forms a \emph{boundary pair} if
the actions $\Gamma\acts B_-$ and $\Gamma\acts B_+$ are amenable, and the projections
\[
	B_-\times B_+\ \overto{}\  B_-,\qquad B_-\times B_+\ \overto{}\  B_+
\]
are relatively isometrically ergodic.
A Lebesgue $\Gamma$-space $B$ for which $(B,B)$ is a boundary pair will be called a $\Gamma$-\emph{boundary}.
\end{definition}

\begin{remarks}\label{R:boundaries}
	\begin{itemize}
		\item[(1)]
		If $(B_-,B_+)$ is a boundary pair for $\Gamma$, then $B_-\times B_+$ is isometrically ergodic.
		This follows by applying Propositions~\ref{P:riser-props}.(iii) and (i) to maps
		\[
			B_-\times B_+\ \overto{}\ B_-,\qquad
		 	B_-\times B_+\ \overto{}\  B_+\ \overto{}\  \{*\}.
	 	\]
		Therefore a $\Gamma$-boundary in the sense of Definition~\ref{D:boundary-pairs}, is also doubly ergodic with unitary coefficients,
		i.e. is a \emph{strong $\Gamma$-boundary} in the sense of Burger-Monod \cite{Burger+Monod}.
		\item[(2)]
		Every lcsc group $\Gamma$ admits boundary/ies in the above sense, see Theorem~\ref{T:FP} below.
		\item[(3)]
		Being a boundary is inherited by lattices: for any lcsc group $G$ any $G$-boundary $B$ is also a
		$\Gamma$-boundary for any lattice $\Gamma<G$.
		\item[(4)]
		Let $B_1$ be a $G_1$-boundary, and $B_2$ be a $G_2$-boundary for some lcsc groups $G_1, G_2$.
		Then $B=B_1\times B_2$ is a $G_1\times G_2$-boundary.
	\end{itemize}
\end{remarks}

%

\subsection*{Some Examples}
Most examples of boundaries that are used in rigidity theory turn out to have the properties stated in Definition~\ref{D:boundary-pairs}.
Let us outline the proofs in two basic cases.

\begin{theorem}\label{T:GmodP-bnd}\hfill{}\\
	Let $G$ be a connected semi-simple Lie group, $P<G$ be a minimal parabolic subgroup.
	Then $B=G/P$ with the Lebesgue measure class is a $G$-boundary, and is a $\Gamma$-boundary for any lattice $\Gamma<G$.
\end{theorem}
%

\begin{proof}
Since $P$ is amenable, $G\acts G/P$ is an amenable action (cf. \cite{Zimmer-amen}). 
So it remains to show that the projection $G/P\times G/P\ \overto{}\ G/P$
is relatively isometrically ergodic. 
Typical (here from the measurable point of view) pairs $g_1P$, $g_2P$ intersect along a coset of the centralizer $A'=\calZ_G(A)$ of a maximal split torus $A<P$.
So as a measurable $G$-space  $G/P\times G/P$ is the same as $G/A'$, and the projection corresponds to the map $gA'\mapsto gP$.

The following is a version of Mautner's Lemma.
\begin{lemma}\label{L:Mautner}
	The $P$-space $P/A'$ is isometrically ergodic.
\end{lemma}
\begin{proof}
	There is a natural correspondence between $P$-equivariant maps $P/A'\to M$
	from the transitive $P$-action on $P/A'$, and the $P$-orbits of $A'$-fixed points $x_0\in M$.
	Mautner's phenomenon in this context, is the statement that in an isometric action
	$P\to \Is(M,d)$ any $A'$-fixed point $x_0$ is fixed also by all elements $u\in P$ for which one can
	find a sequence $a_n\in A'$ with $a_n^{-1}ua_n\to e$.
	Indeed, using continuity of the homomorphism $P\to \Is(M,d)$, for an $A'$-fixed $x_0$ we have
	\[
		d(u.x_0,x_0)=d(ua_n.x_0,a_n.x_0)=d(a_n^{-1} u a_n^{-1}.x_0,x_0)\to d(x_0,x_0)=0.
	\]
	There $x_0$ is fixed by any such $u$.
	The Lemma is proven because the minimal parabolic $P<G$
	is generated by $A'$ and elements $u$ as above, so $x_0$ is $P$-fixed, and the corresponding map is constant.
\end{proof}

The relative isometric ergodicity for the transitive $G$-actions $\pi:G/A'\to G/P$ follows formally from the isometric ergodicity of
a.e. stabilizer ${\rm Stab}_G(gP)$ on its fiber $\pi^{-1}(\{gP\})$, but these are isomorphic to the action $P\acts P/A'$
which is isometrically ergodic by Mautner's Lemma~\ref{L:Mautner}.
This proves that $G/P$ is a $G$-boundary.
This property is inherited by any lattice $\Gamma<G$, so $\Gamma\acts G/P$ is also a boundary action.
\end{proof}

For products of groups $G=G_1\times \cdots\times G_n$ one can use the product of $G_i$-boundaries of the factors
$B=B_1\times \cdots\times B_n$ to obtain a $G$-boundary (Remark~\ref{R:boundaries}.(4)).
Thus this result can be extended to products of semi-simple groups over various fields.

Let us now show that any lcsc group $\Gamma$ has $\Gamma$-boundaries.
Specifically we shall show that the Furstenberg-Poisson boundary for a generating spread-out
random walk on $\Gamma$ forms a boundary pair in the sense of Definition~\ref{D:boundary-pairs}.
This strengthens the result of Kaimanovich \cite{Kaim-DE} showing ergodicity with
unitary coefficients for  $\Gamma\acts B\times \check{B}$ below.

\begin{theorem}[Furstenberg-Poisson boundaries]\label{T:FP}\hfill{}\\
	Let $\Gamma$ be a lcsc group and $\mu$ be a spread-out generating measure on $\Gamma$.
	Denote by $(B,\nu)$ and $(\check{B},\check{\nu})$ the Furstenberg-Poisson boundaries for $(\Gamma,\mu)$ and $(\Gamma,\check{\mu})$.
	Then $(\check{B},B)$ is a boundary pair for $\Gamma$ and for any of its lattices.
	Taking a symmetric spread-out generating $\mu$, the Furstenberg-Poisson boundary $B=\check{B}$ is a $\Gamma$-boundary.
\end{theorem}

\begin{proof}
	Amenability of the actions $\Gamma\acts B$, $\Gamma\acts\check{B}$ being well known (Zimmer \cite{Zimmer-amen}),
	so it remains to prove relative isometric ergodicity.
	It suffices to treat one of the projections, say
	\[
		\pr_B:B\times \check{B}\ \overto{}\ B, \qquad \pr_B(x,y)=x.
	\]
	We shall do so by establishing the following stronger property, whose proof uses a combination of Martingale Convergence Theorem
	for the $\check\mu$-random walk (see (\ref{e:MCT}) below),
	and Poincar\'e recurrence for a non-invertible p.m.p. skew-product (see (\ref{e:OmegaB})).
	\begin{lemma}\label{L:reSAT}
		Given a positive $\nu\times\check\nu$-measure subset $E\subset B\times\check{B}$ and $\epsilon>0$ there is $g\in \Gamma$
		and a positive $\nu$-measure subset $C\subset\pr_B(E)\cap  \pr_B(g^{-1}E)$ so that for $x\in C$
		\[
			\check\nu(g(E_x))>1-\epsilon
		\]
		where $E_x=\{y\in \check{B} \mid (x,y)\in E\}$.
	\end{lemma}
	\begin{proof}
		Denote by $(\Omega,\mu^\bbN)$ the infinite product space $(\Gamma,\mu)^\bbN$.
		The Furstenberg-Poisson boundary $(B,\nu)$ of $(\Gamma,\mu)$ can be viewed as a quotient $\bnd:(\Omega,\mu^\bbN)\to (B,\nu)$,
		where $\bnd(\omega)$ is the limit of the paths of the $\mu$-random walk
		\[
			\bnd(\omega)=\lim_{n\to\infty} \pi_n(\omega),\qquad \pi_n(\omega)=\omega_1\omega_2\cdots\omega_n
		\]
		with the convergence being understood as convergence of values of bounded $\mu$-harmonic functions
		(cf. \cite{Furst-Poisson}, \cite{KV}, \cite{Kaim-DE}).
		The $\check\mu$-boundary $(\check{B},\check{\nu})$ can also be viewed as a quotient of $(\Omega,\mu^\bbN)$, using
		\[	
			\check\bnd(\omega)=\lim_{n\to\infty}\check\pi_n(\omega),\qquad
			\check\pi_n(\omega)= \omega_1^{-1}\omega_2^{-1}\cdots\omega_n^{-1}.
		\]
		By the Furstenberg-Poisson formula, every measurable set $D\subset \check{B}$
		defines a bounded $\check\mu$-harmonic function $h_D:\Gamma\to [0,1]$ by
		\[
			h_D(g)=\int_{\check{B}} 1_D(y) \, dg_*\check\nu(y)=\check\nu(g^{-1}D).
		\]
		Furthermore, by the Martingale Convergence Theorem, for $\mu^\bbN$-a.e. $\omega\in\Omega$ we have
		\begin{equation}\label{e:MCT}
			h_D(\check\pi_n(\omega))\to 1_D(\check\bnd(\omega)).
		\end{equation}
		In particular, the set $\Omega_D=\left\{ \omega\in \Omega \mid h_D(\check\pi_n(\omega))=\nu(\omega_n\cdots \omega_1D)\ \to\ 1\right\}$
		satisfies $\mu^\bbN(\Omega_D)=\check\nu(D)$.
		Given $E\subset B\times\check{B}$ with $\nu\times\check\nu(E)>0$ and $\epsilon>0$, consider the measurable sets
		\begin{align*}
			E^* &=\{ (\omega,x)\in \Omega\times B \mid \omega\in \Omega_{E_x}\},\\
			E^*_N&=\{ (\omega,x)\in E^* \mid \forall n\ge N,\ \ \check\nu(\omega_n\cdots \omega_1E_x)>1-\epsilon \}.
		\end{align*}
		We have
		\[
			\mu^\bbN\times\nu(E^*)=\int_{B} \mu^\bbN(\Omega_{E_x})\,d\nu(x)=\int_{B}\check\nu(E_x)\,d\nu(x)=\nu\times\check\nu(E)>0.
		\]
		Since $E^*_N$ increase to $E^*$, we can find $N$ large enough so that $\nu\times\check\nu(E^*_{N})>0$.
		
		The fact that $\nu$ is $\mu$-stationary implies that the following skew-product transformation
		\begin{equation}\label{e:OmegaB}
			S:(\omega_1,\omega_2,\dots,x)\mapsto (\omega_2,\omega_3,\dots, \omega_1.x)\qquad \textrm{on}\qquad \Omega\times B
		\end{equation}
		preserves the probability measure $\mu^\bbN\times\nu$.
		Therefore, Poincar\'e recurrence implies that we can find
		\[
			n>N\qquad\textrm{so\ that}\qquad \mu^\bbN\times \nu(S^{-n}(E^*_N)\cap E^*_N)>0.
		\]
		Denote $F=S^{-n}(E^*_N)\cap E^*_N$ and let $F_\omega=\{ x\in B \mid (\omega,x)\in F\}$.
		By Fubini, there is a positive $\mu^\bbN$-measure set of $\omega$, for which $\nu(F_\omega)>0$.
		Fix such an $\omega$ and set
		\[
			g=\check\pi_n(\omega)^{-1}=\omega_n\cdots\omega_1,\qquad C=F_\omega.
		\]
		Then $C\subset \pr_B(E^*_N)\subset \pr_B(E)$ and for every $x\in C$ one has $\check\nu(gE_x)>1-\epsilon$.
	\end{proof}
	Let us now complete the proof of the Theorem by showing
	how the property described in the Lemma implies relative isometric ergodicity.
	First consider an arbitrary Borel
	probability measure $\beta$ on a metric space $(M,d)$, and for a small radius $\rho>0$ define
	$w(m,\rho)=\beta({\rm Ball}(m,\rho))$.
	We point out that for $\beta$-a.e. $m\in M$ one has $w(m,\rho)>0$ (this is easier to see for separable spaces).
	Note also that $\beta$ is a Dirac mass $\delta_m$ iff for every $\epsilon>0$
	there exists $m'\in M$ with $w(m',\epsilon)>1-\epsilon$.
	
	Now consider a fiber-wise isometric $\Gamma$-action on some $q:\calM\to V$ and pair of compatible $\Gamma$-maps
	$f:B\times \check{B}\to \calM$, $f_0:B\to V$.
	For $\nu$-a.e. $x\in B$, the pushforward of $\check\nu$ by $f(x,-)$ is a probability measure $\beta_x$ on the
	fiber $q^{-1}(\{f_0(x)\})$ that we shall denote $(M_{x}, d_x)$.
	To construct the required map $f_1:B\to \calM$ we will show that a.e. $\beta_x$ is a Dirac measure and define $f_1$
	by $\beta_x=\delta_{f_1(x)}$. Assuming this is not the case, there exists $\epsilon>0$ and a positive measure
	set $A\subset B$ so that the function
	\[
		w_x(y,\rho)=\beta_x\left({\rm Ball}_{d_x}(f(x,y),\rho)\right)
	\]
	satisfies $w_x(y,\epsilon)<1-\epsilon$ for all $(x,y)\in A\times\check{B}$.
	Since $\check\nu\times\nu$-a.e. $w_x(y,\epsilon)>0$,
	there exists a measurable map $A\to \check{B}$, $x\mapsto y_x$, with $w_x(y_x,\epsilon)>0$.
	Then the set
	\[
		E=\{ (x,z) \in A\times\check{B} \mid z\in {\rm Ball}_{d_x}(f(x,y_x),\epsilon)\}
	\]
	has positive measure, and by Lemma~\ref{L:reSAT}, there is $C\subset \pr_B(E)=A$ and $g\in \Gamma$ so that
	for $x\in C$ one has $g.x\in C\subset A$ and
	\[
		1-\epsilon<\check\nu(g {\rm Ball}_{d_x}(f(x,y_x),\epsilon))
		=\check\nu({\rm Ball}_{d_{g.x}}(f(g.x, g.y_x),\epsilon))=w_{g.x}(g.y_x,\epsilon).
	\]
	This contradiction completes the proof that $(B, \check{B})$ is a $\Gamma$-boundary pair.
\end{proof}

Let us add a purely geometric example.

\begin{example}\label{E:PS}
	Let $M$ be a compact Riemannian manifold of negative curvature, $\partial\tilde{M}$ the boundary
	of the universal cover $\tilde{M}$ of $M$, and let
	$\nu^{\rm PS}_o$ be the Patterson-Sullivan measure relative to some $o\in \tilde{M}$.
	Then $\partial\tilde{M}$ with the Patterson-Sullivan class is a $\Gamma$-boundary for the fundamental group $\Gamma=\pi_1(M)$.
\end{example}

The relative isometric ergodicity in this context can be shown using an analogue of Lemma~\ref{L:reSAT},
whose proof in this case would use Poincar\'e recurrence of the geodesic flow
on the unit tangent bundle $SM$ with Bowen-Margulis-Sullivan measure,
combined with Lebesgue differentiation instead of Martingale convergence used
in the preceding proof.
However, both Example~\ref{E:PS} and Theorem~\ref{T:FP}, can also be established in a different way using
Theorem~\ref{T:wind-boundaries} below.
The proof of the latter is inspired by Kaimanovich \cite{Kaim-DE}.

\section{Characteristic maps}

One of the applications of boundaries is a construction of characteristic maps (a.k.a. boundary maps)
associated to representations of the group, or to cocycles of ergodic p.m.p. actions of the group.
In particular, characteristic maps play a key role in higher rank superrigidity
(see \cite[Chapters V, VI]{Margulis-book}, \cite{Furst-bourbaki}, \cite{Zimmer-book}).
In this section we shall illustrate the use of relative isometric ergodicity by deducing special properties
of characteristic maps in three settings:
for convergence actions, actions on the circle, and linear representations over $\bbR$.
All the results have natural analogues in the context of measurable cocycles
over ergodic p.m.p. actions, but we shall not state these results here.

\subsection*{Convergence actions}

Let $G$ be a lcsc group and $M$ be a compact $G$-space.
For $n\geq 2$ we denote by $M^{(n)}$ the subset of $M^n$ consisting of distinct $n$-tuples, that is
\[
	M^{(n)}=\left\{(m_i)\in M^n \mid ~m_i\neq m_j\ \ \textrm{if}\ \ i\ne j\in \{1,\dots,n\}\right\}.
\]
The $G$-action $G\acts M$ is a \emph{convergence action},
if the diagonal $G$-action on $M^{(3)}$ is proper.
To avoid trivial examples we assume $G$ is not compact and ${\rm card}(M)>2$ (in which case ${\rm card}(M)=2^{\aleph_0}$).
Subgroups of $G$ that stabilize a point, or an unordered pair of points in $M$, are called \emph{elementary}.

Examples of convergence actions include (but not restricted to)
non-elementary groups of isometries of proper $\delta$-hyperbolic spaces acting on their Gromov boundary.
This includes Gromov-hyperbolic groups and their non-elementary subgroups, relatively hyperbolic groups and other examples.
In the case of relatively hyperbolic groups peripheral subgroups are elementary.

\begin{remark}\label{R:conv-min}
	Any convergence action $G\acts M$ has a unique minimal $G$-invariant closed subset $L(G)\subset M$.
	Given a closed subgroup $H<G$, both $H\acts M$ and $H\acts L(H)\subset M$ are convergence actions.
	So given a group $\Gamma$ and a homomorphism $\rho:\Gamma\to G$ with unbounded and non-elementary image
	in a group with a convergence action $G\acts M$,
	upon replacing $G$ by $\overline{\rho(\Gamma)}$ and $M$ by $L(\overline{\rho(\Gamma)})$,
	we may assume $\rho(\Gamma)$ to be dense in $G$ and $G\acts M$ to be a minimal convergence action.
	To avoid trivial situations, one assumes that $\rho(\Gamma)$ is non-elementary and not precompact in $G$.
\end{remark}

\begin{theorem}\label{T:char-conv}
	Let $\Gamma$ be an lcsc group, $(B_+,B_-)$ a boundary pair for $\Gamma$,
	$G\acts M$ a convergence action, and $\rho:\Gamma\to G$ a homomorphism with dense image.
	Assume that $\rho(\Gamma)$ is non-elementary and not precompact in $G$.
	Then there exist $\Gamma$-maps $\phi_+:B_+\to M$, $\phi_-:B_-\to M$ such that the image of
	\[
		\phi_{\bowtie}=\phi_+\times \phi_-:B_-\times B_+\to M^2
	\]
	is essentially contained in $M^{(2)}$ and
	\begin{itemize}
		\item[{\rm (i)}] $\Map_\Gamma(B_-,\Prob(M))=\{ \delta\circ \phi_-\}$,
		$\quad\Map_\Gamma(B_+,\Prob(M))=\{ \delta\circ \phi_+\}$.
		\item[{\rm (ii)}] $\Map_\Gamma(B_-\times B_+,M) = \{ \phi_-\circ \pr_-,\phi_+\circ \pr_+\}$,
		\item[{\rm (iii)}] $\Map_\Gamma(B_-\times B_+,M^{(2)}) = \{ \phi_{\bowtie},\ \tau\circ \phi_{\bowtie}\}$ where $\tau(m,m')=(m',m)$.
	\end{itemize}
\end{theorem}

\begin{proof}[Sketch of the proof]
	Let $G\acts \Sigma$ be a proper action of $G$ on some locally compact separable space (e.g. $\Sigma=M^{(3)}$).
	We claim that 
	\[
		\Map_\Gamma(B_-\times B_+,\Sigma)=\emptyset.
	\]
	Indeed, properness implies that the quotient $\Sigma/G$ is Hausdorff and the stabilizers $K_s={\rm Stab}_G(s)$, $s\in \Sigma$,
	are compact subgroups.
	Ergodicity of $B_-\times B_+$ implies that any $\Gamma$-map $\Psi:B_-\times B_+\to \Sigma$ essentially ranges
	into a single $G$-orbit $G.s_0\cong G/K_{s_0}$.
	It is easy to see that there exists a compact subgroup $K<G$ such that  $K_{s_0}<K$ and $G/K$ carries a $G$-invariant metric
	(e.g, $K$ is the stabilizer of a $K_{s_0}$-invariant positive function in $L^2(G)$).
	We obtain a $\Gamma$-invariant map $B_-\times B_+\to G/K$, which is constant by isometric ergodicity,
	and we conclude that $\rho(\Gamma)$ is contained in a conjugate of $K$.
	Thus the existence of such $\Psi$ contradicts the assumption that $\rho(\Gamma)$ is not precompact.

	By amenability of $\Gamma\acts B_\pm$ we may choose $\Gamma$-maps $\Phi_\pm\in\Map_\Gamma(B_\pm,\Prob(M))$.
	Consider the function $\Psi\in \Map_\Gamma(B_-\times B_+, \Prob(M^3))$ defined by
	\[
		\Psi:(x,y)\mapsto \Phi_-(x)\times\Phi_+(y)\times \frac{\Phi_-(x)+\Phi_+(y)}{2}\in \Prob(M^3).
	\]
	Since $G\acts M^{(3)}$ is proper, the action $G\acts \Prob(M^{(3)})$ is also proper,
	and using the above argument with $\Sigma=\Prob(M^{(3)})$,
	it follows that $\Psi$ is supported on the big diagonal $\Delta_{12}\cup\Delta_{23}\cup\Delta_{31}$,
	where $\Delta_{ij}=\{(m_1,m_2,m_3)\in M^3 \mid m_i=m_j\}$.
	This implies that for a.e. $(x,y)$ the measure $\Psi(x,y)$ is atomic with at most two atoms,
	and consequently that $\Phi_-(x)$ and $\Phi_+(y)$ must be Dirac measures.
	We define $\phi_-$ by $\Phi_-(x)=\delta_{\phi_-(x)}$, and $\phi_+$ by $\Phi_+(y)=\delta_{\phi_+(y)}$.
	We also conclude that the essential image of any $\Gamma$-map $B_+\to \Prob(M)$ consists of $\delta$-measure.
	It follows that such a map is unique:
	indeed, given $\Phi'_-:B_-\to \Prob(M)$ we may also consider the map
	\[
		x\mapsto \frac{\Phi_-(x)+\Phi'_-(x)}{2}
	\]
	and conclude that a.e $\Phi'_-(x)=\Phi_-(x)$. A similar argument applies to give uniqueness of $y\mapsto \delta_{\phi+(y)}$
	as an element of $\Map_\Gamma(B_+,\Prob(M))$.
	This proves (i).

	Given $\psi\in \Map_\Gamma(B_-\times B_+,M)$, consider $\Psi=\psi\times (\phi\circ\pr_-)\times (\phi\circ\pr_+)$ as a $\Gamma$-map
	$B_-\times B_+\to M^3$. Since $\Gamma\acts M^{(3)}$ is a proper action, it follows that $\Psi$ takes values in $M^3\setminus M^{(3)}$,
	and more specifically in $\Delta_{12}$ or in $\Delta_{13}$ (because $\Delta_{23}$ is impossible), which gives (ii), while (iii)
	easily follows from (ii).
\end{proof}
Note that this proof used only the amenability of $B_-$, $B_+$ and isometric ergodicity of their product, $B_-\times B_+$,
but did not rely on the relative isometric ergodicity of the projections.

\subsection*{Actions on the circle}

Consider an action of some group $\Gamma$ on a circle $S^1$ by homeomorphisms.
Up to passing to an index two subgroup, we may assume the action to be orientation preserving,
and obtain a homomorphism $\Gamma\to \Homeo_+(S^1)$.
Hereafter we shall assume that the action has no finite orbits.
It is well known that in such case $\Gamma$ has a unique minimal set $K\subset S^1$, and either $K=S^1$, or $K$ is a Cantor set.
In the latter case, $S^1\setminus K$ is a countable dense set of open arcs;
collapsing these arcs one obtains a degree one map $h:S^1\to S^1$ that intertwines
the given $\Gamma$-action with a minimal $\Gamma$-action on the circle.

Given a minimal $\Gamma$-action on the circle the following dichotomy holds
(see Margulis \cite{Margulis-circle}, Ghys \cite{Ghys-circle}):
either $\rho(\Gamma)$ is equicontinuous, in which case it is conjugate into rotation group ${\rm SO}(2)$,
or the centralizer $Z$ of $\rho(\Gamma)$ in $\Homeo_+(S^1)$ is a finite cyclic group,
and the $|Z|$-to-$1$ cover $g:S^1\to S^1/Z$ intertwines the given minimal $\Gamma$-action with a
\emph{minimal and strongly proximal}\footnote{In general, a $\Gamma$-action $\Gamma\acts M$ on a compact metrizable $M$ is minimal and strongly proximal if
for every $\nu\in\Prob(M)$ the closure $\overline{\Gamma.\nu}\subset\Prob(M)$ contains $\delta_M=\{\delta_x \mid x\in M\}$.} one,
where the latter term can be taken to mean that for any proper closed arc $J\subsetneq S^1$ and any non-empty open arc
$U\ne\emptyset$ there is $g\in\Gamma$ with $gJ\subset U$.
To sum up, any group action with only infinite orbits is semi-conjugate either to rotations, or to a minimal and strongly proximal action.
We shall focus on the latter class of actions.
\begin{theorem}\label{T:char-circle}
	Let $\Gamma\to\Homeo_+(S^1)$ be a minimal and strongly proximal action on the circle,
	and let $(B_-,B_+)$ be a boundary pair for $\Gamma$.
	Then there exist $\Gamma$-maps $\phi_+:B_+\to S^1$, $\phi_-:B_-\to S^1$ such that the image of
	\[
		\phi_{\bowtie}=\phi_-\times \phi_+:B_-\times B_+\to (S^1)^2
	\]
	is essentially contained in  the space of distinct pairs $(S^1)^{(2)}$, and
	\begin{itemize}
		\item[{\rm (i)}] $\Map_\Gamma(B_-,\Prob(S^1))=\{ \delta\circ \phi_-\}$,
		$\quad\Map_\Gamma(B_+,\Prob(S^1))=\{ \delta\circ \phi_+\}$.
		\item[{\rm (ii)}] $\Map_\Gamma(B_-\times B_+,S^1)$ has a canonical cyclic order.
	\end{itemize}
\end{theorem}
\begin{proof}[Sketch of the proof]
	Following Ghys \cite{Ghys-circle} we note that
	\[
		d(\mu_1,\mu_2)=\max\left\{|\mu_1(J) - \mu_2(J)|\,:\, J\subset S^1\ \textrm{is\ an\ arc}\right\},
	\]
	is a $\Homeo_+(S^1)$-invariant metric on the subspace $\Prob_c(S^1)$ of all continuous (i.e. atomless) probability measures on $S^1$.
	This shows that there are no $\Gamma$-maps from any isometrically ergodic $\Gamma$-space $A$ to $\Prob_c(S^1)$,
	because $\Prob_c(S^1)^\Gamma=\emptyset$ under the assumption of
	minimality and strongly proximality of $\rho(\Gamma)$.
	
	By amenability there exist $\Phi_\pm\in \Map_\Gamma(B_\pm,\Prob(S^1))$, and the above argument shows that they take values in
	atomic measures. We claim that $\Phi_\pm=\delta_{\phi_\pm}$ for some unique $\phi_\pm\in \Map_\Gamma(B_\pm,S^1)$.
	Indeed, fix $w>0$ and let $A_-(x)=\{ a\in S^1 \mid \Phi_-(x)(\{a\})>w\}$ denote the set of atoms of $\Phi_-(x)$ of weight $\ge w$,
	and define $A_+(y)$ similarly.
	Then, for $w>0$ small enough, $x\mapsto A_-(x)$ is a $\Gamma$-equivariant assignment of non-empty finite subsets of $S^1$;
	and by ergodicity the cardinality of $A_-(x)$ is a.e. constant $k\in\bbN$.
	Similarly for $A_+(y)$.
	
	Let us say that $(x,y)$ is a \emph{good pair} if  $A_-(x)$ is \emph{unlinked} with $A_+(y)$, i.e. they belong to disjoint arcs.
	Since the set of good pairs is $\Gamma$-invariant, it is either null or conull in $B\times\check{B}$ by ergodicity.
	Choose proper closed arcs $I,J\subsetneq S^1$ so that $E=\{ x\in B \mid A_-(x)\subset J\}$ has $\nu(E)>0$,
	and $F=\{ y \mid A_+(y)\subset I\}$ has $\check\nu(F)>0$.
	By minimality and strong proximality, there exists $g\in\Gamma$ with $gJ\cap I=\emptyset$.
	Then $gE\times F$ is a positive measure set of good pairs. Hence for a.e. $(x,y)$ the sets $A_-(x)$ and $A_+(y)$ are unlinked.
	
	For a.e. fixed $x\in B$, the complement $S^1\setminus A_-(x)=\bigsqcup_{i=1}^k U_i(x)$ is a disjoint union of $k$ open arcs $U_i(x)$,
	where the enumeration is cyclic and $x\mapsto U_1(x)$ can be assumed to be measurable.
	The $\Gamma$-action cyclically permutes these intervals:
	$\rho(g) U_i(x)=U_{\pi_x(i)}(g.x)$ by some $\pi_x\in {\rm Sym}_k$.
	The fact that $A_-(x)$ is unlinked from a.e. $A_+(y)$ means that $A_+(y)\in U_{i(x,y)}(x)$
	for some measurable $i:B_-\times B_+\to \{1,\dots,k\}$,
	while relative isometric ergodicity of $B_-\times B_+\to B_-$ implies that $i(x,y)=i(x)$ is essentially independent of
	$y\in B_+$.
	Thus for a.e. $x\in B_-$, the closure $J(x)=\overline{U_{i(x)}(x)}$ contains a.e. $A_+(y)$.
	By minimality of the $\Gamma$-action on $S^1$, it follows that $J(x)=S^1$ and $k=1$.
	As $w>0$ was arbitrary, it follows that $\Phi_-(x)=\delta_{\phi_-(x)}$
	for some $\phi_-\in\Map_\Gamma(B_-,S^1)$.
	Similarly we get $\Phi_+=\delta\circ \phi_+$ for a unique $\phi_+\in\Map_\Gamma(B_+,S^1)$.
	
	For a.e. $(x,y)\in B_-\times B_+$ the given orientation of $S^1$ defines a cyclic order on
	every triple in $\Map_\Gamma(B_-\times B_+,S^1)$ by evaluation. This order is $\Gamma$-invariant,
	and is therefore a.e. constant by ergodicity. This shows (ii).
\end{proof}

It is possible that $\Map_\Gamma(B_-\times B_+,S^1)=\{\phi_-\circ\pr_-, \phi_+\circ\pr_+\}$; but
short of proving this we will rely on the cyclic order (ii) that would suffice for our arguments.
We note also that the concept of relative isometric ergodicity allows to improve the argument
from \cite{BFS-circle} that was based only on double ergodicity.
See also \cite[\S 2]{Burger-4Zimmer} for a different argument that gives the above result 
in the special case of $B$ being a Furstenberg-Poisson boundary.

\subsection*{Linear representations}

Let $G$ be a connected, center-free, simple, non-compact, real Lie group, $P<G$ a minimal parabolic subgroup,
$A<P$ a maximal split torus, and $A'=\calZ_G(A)$ its centralizer.
Since $A'<P$ one has a natural $G$-equivariant projection $\pr_1:G/A'\to G/P$.
The Weyl group of $G$ can be defined as $\calN_G(A)/\calZ_G(A)=\calN_G(A')/A'=\Aut_G(G/A')$; and can also be used
to parameterize $\Map_G(G/A',G/P)$.
If $\wlong\in \cW{G}$ denotes the \emph{long element} of this Coxeter group,
then $\pr_2=\pr_1\circ \wlong:G/A'\to G/P$ is the opposite projection, so that
\[
	\pr_1\times \pr_2:G/A'\to G/P\times G/P
\]
is an embedding, whose image is the \emph{big} $G$-orbit.
For $G=\PSL_d(\bbR)$, $A<P$ are the diagonal and the upper triangular subgroups,
$G/A'$ is the space of $d$-tuples $(\ell_1,\dots,\ell_d)$
of $1$-dimensional subspaces that span $\bbR^d$,
$\cW{G}\cong {\rm Sym}_d$ acts by permutations,
$\wlong$ is the order reversing involution $j\mapsto (n+1-j)$,
$G/P$ is the space of flags $(E_1,\dots,E_d)$ consisting of nested vector subspaces $E_j<E_{j+1}$ with $\dim(E_j)=j$,
and
\[
	\begin{split}
	&\pr_1:(\ell_1,\dots,\ell_d)\mapsto (\ell_1,\ell_1\oplus \ell_2, \dots,\ell_1\oplus\cdots\oplus \ell_d=\bbR^d),\\
	&\pr_2:(\ell_1,\dots,\ell_d)\mapsto (\ell_d,\ell_{d-1}\oplus \ell_d, \dots,\ell_1\oplus\cdots\oplus \ell_d=\bbR^d).
	\end{split}
\]
The image of $\pr_1\times \pr_2$ consists of pairs of flags that are in a general position.

If ${\rm rank}_\bbR(G)=1$, then one can identify $G/A'$ with the space of oriented but unparameterized geodesic lines
in the symmetric space $X$ of $G$,
$G/P$ with sphere at infinity $\partial_\infty X$, $G/A'\to G/P$ associating the limit at $-\infty$ of the geodesic, and
$\cW{G}\cong \bbZ/2\bbZ$ acting by flipping the orientation/ endpoints of the geodesics.
The image of $G/A'$ in $G/P\times G/P$ consists of all distinct pairs.
In this case $G\acts G/P$ is a convergence action, and in the collection of subgroups of $G$ which act minmally on $G/P$, the Zariski dense subgroups are precisely the ones that are not precompact and non-elementary.
Hence the following result in the special case of ${\rm rank}_\bbR(G)=1$ can also be deduced from Theorem~\ref{T:char-conv}.

\begin{theorem}\label{T:char-alg}
	Let $\Gamma$ be an lcsc group, $(B_+,B_-)$ a boundary pair for $\Gamma$, $G$ a non-compact connected simple Lie group,
	and $\rho:\Gamma\to G$ a homomorphism.
	Assume that $\rho(\Gamma)$ is Zariski dense in $G$.
	Then there exist $\Gamma$-maps $\phi_-:B_-\to G/P$, $\phi_+:B_+\to G/P$ and $\phi_{\bowtie}:B_-\times B_+\to G/A'$
 	such that
	\[	
		\pr_1\circ \phi_{\bowtie}(x,y)=\phi_-(x),\qquad \pr_2\circ \phi_{\bowtie}(x,y)=\phi_+(y)\qquad(x\in B_-,\,y\in B_+),
	\]
	and
	\begin{itemize}
		\item[{\rm (i)}] $\Map_\Gamma(B_-,\Prob(G/P))=\{ \delta\circ\phi_-\}$,
		            $\quad\Map_\Gamma(B_+,\Prob(G/P))=\{ \delta\circ \phi_+\}$,
		\item[{\rm (ii)}] $\Map_\Gamma(B_-\times B_+,G/P) = \{ \pr_1\circ w\circ\phi_{\bowtie}~|~w\in \cW{G}\}$,
		\item[{\rm (iii)}] $\Map_\Gamma(B_-\times B_+,G/A') = \{ w\circ \phi_{\bowtie}~|~w\in \cW{G}\}$.
	\end{itemize}
\end{theorem}

As we have mentioned the above theorem is aimed at higher rank target groups $G$, where $\cW{G}$ has more than just $\{e,\wlong\}$.
In the forthcoming paper \cite{BF-bnd} we prove a general version of  Theorem~\ref{T:char-alg}, which is valid for algebraic groups
$G$ defined over an arbitrary local field (in fact over any spherically complete field).
The proof uses the formalism of representations of ergodic actions, developed in our recent paper \cite{BF-csr}.
There we show that for every ergodic Lebesgue $\Gamma$-space $X$ there exists an algebraic subgroup $H<G$
and $\phi\in\Map_\Gamma(X,G/H)$ having the following universal property:
for every $G$-variety $V$ and $\psi\in\Map_\Gamma(X,V)$ there exists a $G$-algebraic morphism $\pi:G/H\to V$ so that $\psi=\pi\circ \phi$ a.e. on $X$
(this is closely related to Zimmer's notion of algebraic hull, see also \cite{Burger-ACampo}).
We apply this result to our setting and let $\phi_+:B_+\to G/H_+$, $\phi_-:B_-\to G/H_-$ and $\phi_0:B_+\times B_-\to G/H_0$ be the corresponding
universal $\Gamma$-maps.
Theorem~\ref{T:char-alg} follows easily once we show that $H_+=H_-=P$ and $H_0=A'$ up to conjugations.
It is precisely this point, where relative isometric ergodicity of $\pr_\pm:B_-\times B_+\to B_\pm$ is used.

\begin{proof}[Sketch of the proof]
We first explain that the amenability of $B_+$ implies that $H_+$ is amenable.
Indeed, there exists a boundary map $B_+\to\Prob(G/P)$ and the ergodicity of $B_+$ implies that its image is essentially contained in a unique $G$-orbit,
as the $G$-orbits on $\Prob(G/P)$ are locally closed \cite{Zimmer-book}.
We get a map $B_+\to G.\mu\simeq G/G_\mu$ for some $\mu\in\Prob(G/P)$.
The stabilizer in $G$ of any point  of $\Prob(G/P)$ is amenable and algebraic (we work over $\bbR$).
In particular, $V=G/G_\mu$ is algebraic, and by the universal property of $H_+$ there exists a $G$-map $G/H_+\to G/G_\mu$.
Thus, up to conjugation, $H_+<G_\mu$. In particular, $H_+$ is amenable.
Similarly $H_-$ is amenable.

Considering the composed map $\phi:B_-\times B_+\to B_+\to G/H_+=V$ and using the universal property of $\phi_0:B_-\times B_+ \to G/H_0$,
we get a $G$-map $\pi:G/H_0\to G/H_+$ such that $\pi\circ\phi_0=\phi_+\circ \pr_+$.
We assume, as we may (by conjugating), that $H_0<H_+$. 
Denoting by $R_+$ the unipotent radical of $H_+$, 
we obtain the containment $H_0<H_0 R_+<H_+$, and the corresponding $G$-maps $G/H_0\to G/H_0R_+\to G/H_+$.
We get the following commutative diagram
	\[
		\xymatrix{
		B_-\times B_+ \ar[d]_{\pr_+} \ar[r] & G/H_0R_+ \ar[d]^q\\
		B_+ \ar[r]^{\phi_+} \ar@{.>}[ur]^{\psi} & G/H_+
		}
	\]
in which the existence of the map $\psi$ is guaranteed by the isometric ergodicity of $\pr_+$.
Indeed, $q$ is fiber-wise $\Gamma$-isometric as its fibers could be identified with $(H_+/R_+)/(H_0\cap R_+)$
which carries an $H_+/R_+$-invariant metric, as the latter group is compact by abelian (since it is reductive and amenable).
By the universal property of $\phi_+$, we conclude that $q$ is an isomorphism.
We therefore obtain $H_+=H_0R_+$.
Similarly, denoting by $R_-$ the unipotent radical of $H_-$ we obtain $H_-=H_0R_-$.

The composed $\phi_+\times \phi_-:B_+\times B_-\to G/H_0\to G/H_+\times G/H_-$ is $\Gamma\times \Gamma$-equivariant,
hence the Zariski closure of its essential image is $\rho(\Gamma)\times \rho(\Gamma)$-invariant.
Since $\rho(\Gamma)$ is Zariski dense in $G$, it follows that the image of $G/H_0$ is Zariski dense in $G/H_+\times G/H_-$.
Equivalently, the set $R_+H_0R_-$ is Zariski dense in $G$.
At this point the proof reduces to the following algebraic group theoretic lemma.

\begin{lemma}
Let $G$ be a reductive group, $H_0<H_+,H_-<G$ algebraic subgroups.
Assume that $H_+=H_0R_+$, $H_-=H_0R_-$ and $R_+H_0R_-$ is Zariski dense in $G$,
where $R_+$ and $R_-$ are the unipotent radicals of $H_+$ and $H_-$ correspondingly.
Then $H_+$ and $H_-$ are opposite parabolics in $G$, and $H_0$ is their intersection.
\end{lemma}

Finally, by the amenability of $H_+$, $H_-$ these parabolics must be minimal in $G$,
$H_0$ conjugate to $A'$, and the result follows.
\end{proof}

\section{Applications to Rigidity} 
\label{sec:rigidity}\hfill{}\\
Let us now demonstrate how boundary theory can be used
to obtain restrictions on linear representations,
convergence actions, and actions on the circle.
These results are inspired by the celebrated Margulis' superrigidity \cite{Margulis-ICM, Margulis-book},
and the developments that followed, including \cite{Zimmer-book,  Burger+Mozes, Ghys-lattices}.
Our aim is to illustrate the techniques rather than to obtain most general results,
in particular we do not state the cocycle versions of the results that can be obtained by similar methods.

\subsection*{Convergence action of a lattice in a product}

Consider a homomorphism $\rho:\Gamma\to G$ where $G\acts M$ is a convergence action.
In view of Remark~\ref{R:conv-min}, we may assume $\rho(\Gamma)$ is dense in $G$ and $G\acts M$ is a minimal convergence action,
${\rm card}(M)=2^{\alpha_0}$ and $G$ is non-compact.
\begin{theorem} \label{T:conv-prod}
Let $S=S_1\times \cdots \times S_n$ be a product of lcsc groups, $\Gamma<S$ a lattice, such that
$\pr_i(\Gamma)$ is dense in $S_i$ for each $i\in\{1,\dots,n\}$.
Assume $G\acts M$ is a minimal convergence action, $G$ is not compact, ${\rm card}(M)>2$,
and $\rho:\Gamma\to G$ is a continuous homomorphism with a dense image.
Then for some $i\in\{1,\ldots,n\}$ there exists a continuous homomorphism $\bar{\rho}:S_i\to G$
such that $\rho=\bar{\rho}\circ \pr_i$.
\end{theorem}

\begin{proof}[Sketch of the proof for $n=2$]
Choose a boundary $B_i$ for each $S_i$, for example using Theorem~\ref{T:FP}, and set $B=B_1\times B_2$.
Then $B$ is an $S$-boundary and a $\Gamma$-boundary (Remark~\ref{R:boundaries}).
By Theorem~\ref{T:char-conv} we have a unique $\Gamma$-map $\phi:B\to M$.
Consider the map
\[
	\Phi:B\times B=B_1\times B_2 \times B_1\times B_2 \ \overto{}\ M^2,\quad
	(x,y,x',y') \mapsto (\phi(x,y),\phi(x,y')).
\]
By Theorem~\ref{T:char-conv}(iii) we have three cases: $\Phi(B\times B)$ is contained in the diagonal $\Delta\subset M^2$,
$\Phi=\phi_{\bowtie}$, or $\Phi=\tau\circ\phi_{\bowtie}$, where $\phi_{\bowtie}=\phi\times \phi$ and $\tau(m,m')=(m',m)$.
In the first case we see that $\phi(x,y)$ is independent of $y\in B_2$, and therefore descends to a $\Gamma$-map $B_1\to M$.
In the second case, $\phi$ is independent of $x\in B_1$, and descends to $B_2\to M$.
The third case gives that $\phi$ is independent of both parameters, thus its essential image is a $\Gamma$-fixed point in $M$.
This is incompatible with $\rho(\Gamma)$ being non-elementary.
We conclude that for some $i\in\{1,2\}$, $\phi:B_1\times B_2\to M$ factors through $B_i$.
We shall apply the following general lemma, letting $X=B_i$, and $\Lambda=\pr_i(\Gamma)$, which is dense in $T=S_i$.

\begin{lemma}\label{L:extension}
	Let $T$ be a lcsc group, $\Lambda<T$ a dense subgroup, $M$ a compact metrizable space, $G<\Homeo(M)$ a closed subgroup and
	$\rho:\Lambda\to G$ a homomorphism.
	Assume that there exists a measurable $T$-space $(X,\mu)$, a $\Lambda$-map $\phi:X\to M$ so that
	the $G$-action on $M$ with $\eta=\phi_*\mu\in\Prob(M)$ satisfies the following condition:
	\begin{itemize}
		\item[${\rm (*)}$] 
		a sequence $\{g_n\}$ in $G$ satisfies $g_n\to e$ in $G$ if (and only if)
			\begin{equation}\label{e:gnhk}
				\int_M(h\circ g_n-h)\cdot k\,d\eta\to 0\qquad (h,k\in C(M))
			\end{equation}
	\end{itemize}
	Then $\rho:\Lambda\to G$ extends to a continuous homomorphism $\bar{\rho}:T\to G$.
	
	Furthermore, the combination of the following two conditions imply the condition (*) defined above:
	\begin{itemize}
		\item[${\rm (*)}_1$]
		$(g_n)_*\eta\to \eta$ in weak-* topology $\quad\Longrightarrow\quad \{g_n\}$ is bounded in $G$,
		\item[${\rm (*)}_2$]
		$\forall g\in G\setminus\{e\}$, $\exists k,h\in C(M)$ so that $\int_M (h\circ g-h)\cdot k\,d\eta \ne 0$.
	\end{itemize}
\end{lemma}
\begin{proof}
	Since $\Lambda$ is dense in $T$, existence of a continuous extension $\bar{\rho}:T\to G$ is equivalent to
	showing that $g_n=\rho(\lambda_n)\to e$ in $G$ for every sequence $\{\lambda_n\}$ in $\Lambda$ with $\lambda_n\to e$ in $T$.
	The $T$-action by pre-composition on $L^\infty(X,\mu)$, equipped with the weak-* topology from $L^1(X,\mu)$,
	is continuous.
	Take $\Lambda\ni\lambda_n\to e$ in $T$, functions $h,k\in C(M)$, and define $\tilde{h}\in L^\infty(X,\mu)$,
	$\tilde{k}\in L^1(X,\mu)$ by $\tilde{h}=h\circ \phi$, $\tilde{k}=k\circ \phi$.
	Then
	\[
			\int_M(h\circ \rho(\lambda_n)-h)\cdot k\,d\eta=\int_X (\tilde{h}\circ \lambda_n-\tilde{h})\cdot \tilde{k}\,d\mu\to0.
	\]
	Hence (\ref{e:gnhk}) detects the convergence $\rho(\lambda_n)\to e$ in $G$.
	
	To see that ${\rm (*)}_1+{\rm (*)}_2\ \Longrightarrow\ {\rm (*)}$, note that for any sequence $\{g_n\}$ in $G$ with (\ref{e:gnhk}),
	there is weak-* convergence $(g_n)_*\eta\to \eta$ by taking $k=1$ and varying $h\in C(M)$.
	Thus $\{g_n\}$ is precompact in $G$ by ${\rm (*)}_1$.
	Condition ${\rm (*)}_2$ implies that $e$ is the only possible limit point for $\{g_n\}$.
	This completes the proof of the Lemma.
\end{proof}
To complete the proof of Theorem~\ref{T:conv-prod} we check conditions ${\rm (*)}_1$, ${\rm (*)}_2$.
In our context $\supp(\eta)=M$, because $\supp(\eta)$ is a $\rho(\Gamma)$-invariant closed subset of $M$,
while $\rho(\Gamma)$ is dense in $G$ and $G\acts M$ is minimal.
This implies ${\rm (*)}_2$.
For ${\rm (*)}_1$ observe that it follows from the convergence property of $G\acts M$
that if $g_i\to\infty$ in $G$ and $(g_i)_*\eta\to \xi\in\Prob(M)$, then $\xi$ is supported on one or two points,
while $\supp(\eta)=M$ is a continuum.
Therefore the conditions of Lemma are satisfied, and we get a continuous extension $\bar\rho:S_i\to G$ as claimed.
\end{proof}

\subsection*{Weyl groups}

The argument showing that the map $\phi\in\Map_\Gamma(B,M)$ factors through one of the boundaries $B_i$ in the proof of Theorem~\ref{T:conv-prod},
might appear to be ad-hoc. But in fact, it can be maid conceptual as follows.
Given a group $\Gamma$ and a choice of a $\Gamma$-boundary, define the associated \emph{generalized Weyl group} to be
\[
	\gW{\Gamma}{B}=\Aut_\Gamma(B\times B),
\]
the group of measure class preserving automorphisms of $B\times B$ that commute with $\Gamma$.
For non-amenable $\Gamma$, a $\Gamma$-boundary cannot be trivial, so $\gW{\Gamma}{B}$ always contains
the non-trivial involution $\wflip:(x,y)\mapsto (y,x)$.

\begin{example} \label{ex:weyl-prod}
For a boundary which is a product of $\Gamma$-spaces, $B=\prod_{i\in I} B_i$, the generalized Weyl group contains a
subgroup isomorphic to $\prod_{i\in I} \bbZ/2\bbZ$ obtained by flipping factors of $B\times B \simeq \prod_{i\in I} (B_i\times B_i)$.
\end{example}

Given a Borel $\Gamma$-space $V$, $\gW{\Gamma}{B}$ acts on $\Map_\Gamma(B\times B,V)$ by precompositions.
For any $\Gamma$-map $\phi:B \to V$ we obtain a subgroup of $\gW{\Gamma}{B}$ - the stabilizer of
$\phi\circ \pr_+\in \Map_\Gamma(B\times B,V)$ under this action.
The subgroups obtained this way are called \emph{special subgroups}.

It is easy to check that the special subgroups of $\prod_{i\in I} \bbZ/2\bbZ$ in Example~\ref{ex:weyl-prod} are the subgroups
of the form $\prod_{i\in J} \bbZ/2\bbZ$ for $J\subset I$.
In the setting of convergence actions,
Theorem~\ref{T:char-conv}(ii) shows that the action of $\gW{\Gamma}{B}$ on $\Map_\Gamma(B\times B,M)$
must factor trough a group of order two. 
The kernel of this action is clearly a special subgroup, it is the stabilizer of $\phi\circ \pr_+$,
and one deduces that
$\phi:\prod B_i\to M$ factors through $B_i\to M$ for some $i\in I$.
Invoking now Lemma~\ref{L:extension}, one obtains this way an alternative proof of Theorem~\ref{T:conv-prod}.

Considering now a lattice in a product of groups acting on a circle, we may apply a similar strategy.
By Theorem~\ref{T:char-circle}(ii) the action of $\gW{\Gamma}{B}$ on $\Map_\Gamma(B\times B,S^1)$ factors through a cyclic subgroup.
Considering again the subgroup $\prod_{i\in I} \bbZ/2\bbZ$ and its special subgroups, we conclude that
$\phi:\prod B_i\to S^1$ factors through $B_i\to S^1$ for some $i\in I$.
The extension Lemma~\ref{L:extension} applies to $G=\Homeo_+(S^1)$, and one deduces the following.

\begin{theorem}\label{T:circle-prod}
	Let $\Gamma$ be an irreducible lattice in a product $S=S_1\times \cdots\times S_n$ of lcsc groups  as in Theorem~\ref{T:conv-prod}.
	Let $\rho:\Gamma\to \Homeo_+(S^1)$ be a minimal strongly proximal action on the circle.
	Then, $\rho$ extends to a continuous homomorphism that factors through some $S_i$,
	namely $\bar\rho_i:S_i\to \Homeo_+(S^1)$ so that $\rho=\bar\rho_i\circ \pr_i$.
	Moreover, if $\bar\rho_i(S_i)$ is non-discrete, then it could be conjugated to $\PSL_2(\bbR)<\Homeo_+(S^1)$,
	so $\Gamma$ may be assumed to act via fractional linear transformations.
\end{theorem}

The addendum about $\PSL_2(\bbR)$ follows from the general fact that a lcsc group acting minimally and strongly proximally
on the circle is either discrete, or could be conjugated into $\PSL_2(\bbR)$.
Let us also remark, that under some mild assumptions
(e.g. $\Gamma$ is finitely generated and projects injectively to the $S_i$-factors)
one can prove that if $\rho(S)$ is non-discrete then, up to finite index and a compact factor,
$\Gamma<S$ is an arithmetic lattice in a finite product
of a real and possible $p$-adic algebraic groups, one of which is $\PSL_2(\bbR)$, and $\Gamma$ acts on the
circle through this factor \cite{BFS-env}.

Next consider a connected simple Lie group $S$ with ${\rm rank}_\bbR(S)\ge 2$.
Let $B=S/Q$ where $Q<S$ is a minimal parabolic. It is an $S$-boundary by Theorem~\ref{T:GmodP-bnd}.
As a measurable $S$-space, $B\times B=S/Q\times S/Q\cong S/A'$, where $A'$ is the centralizer of
a maximal split torus $A<Q$. The generalized Weyl group $\gW{S}{S/Q}$,
consisting of automorphisms of $S/A'$ as a measurable $S$-space, 
is easily seen to coincide with the classical Weyl group $\cW{S}=\calN_S(A)/\calZ_S(A)=\calN_S(A')/A'$:
\[
	\gW{S}{S/Q}\cong \cW{S}.
\]
Let $\Gamma<S$ be a lattice. Then $B=S/Q$ is also a $\Gamma$-boundary and $\gW{\Gamma}{S/Q}$ contains $\cW{S}$.
This inclusion is an isomorphism (\cite{BFGW}) and the three notions of special subgroups: of $\gW{\Gamma}{S/Q}$,
of $\gW{S}{S/Q}$, and of $\cW{S}$ seen as a Coxeter group, all coincide \cite{BF-hyp}.
Since $S$ is assumed to be simple, the Coxeter group $\cW{S}$ is irreducible.
It is not hard to see that if $W$ is an irreducible Coxeter group and $W'<W$ a proper special subgroup then the action of $W$ on the coset space $W/W'$ is faithful: $W'$ contains no nontrivial subgroup which is normal in $W$.
It follows that for any Borel $\Gamma$-space $V$ with $V^\Gamma=\emptyset$, the action of $\gW{\Gamma}{S/Q}$ 
on the orbit of $\phi\circ \pr_+\Map_\Gamma(S/Q\times S/Q,V)$ for $\phi\in\Map_\Gamma(S/Q,V)$ is faithful.
This allows to deduce the following result of Ghys \cite{Ghys-lattices}.

\begin{theorem}
	Let $\Gamma$ be a lattice in a connected simple Lie group $S$ with ${\rm rank}_\bbR(S)\ge 2$.
	Then any $\Gamma$-action on the circle has a finite orbit.
\end{theorem}

Indeed, assuming $\Gamma$ has an action on the circle with only infinite orbits, one could find a minimal such action by applying a semiconjugation.
Since $\Gamma$ cannot act minimally by rotations (because $\Gamma/[\Gamma,\Gamma]$ is finite), it would also have
a minimal strongly proximal action.
Theorem~\ref{T:char-circle}(ii) then guarantees that the action of $\cW{S}\simeq \gW{\Gamma}{S/Q}$
on $\Map_\Gamma(S/Q\times S/Q,S^1)$ factors through a cyclic quotient, contradicting its faithfulness
because $\cW{G}$ is not cyclic for a higher rank $G$.

Similarly, we have the following result that might be seen as a generalization of the special case of
Margulis superrigidity stating that any homomorphism from a higher rank lattice $\Gamma$
into a rank one group has precompact image.

\begin{theorem}\label{T:higher2hyp}
	Let $\Gamma$ be a lattice in a connected simple Lie group $S$ with ${\rm rank}_\bbR(S)\ge 2$.
	Then for any homomorphism $\rho:\Gamma\to G$ where $G\acts M$ is a non-trivial convergence action,
	$\rho(\Gamma)$ is elementary or precompact in $G$.
\end{theorem}

In \cite{BF-hyp} the basic idea of the last result is developed further for a class of target groups that includes
mapping class groups, automorphism groups of finite dimensional CAT(0) cubical complexes etc.

Taking the target group $G$ to be a connected simple real Lie group one can use Theorem~\ref{T:char-alg}
to obtain the following.

\begin{theorem}
Let $\Gamma$ be a lcsc group, $B$ a $\Gamma$-boundary, $G$ a connected, non-compact, simple real Lie group,
and $\rho:\Gamma\to G$ a homomorphism with Zariski dense image.
Then there exists a homomorphism
\[
	\pi:\gW{\Gamma}{B}\ \overto{}\ \cW{G},\qquad \pi(\wflip)=\wlong,
\]
satisfying that the preimage of a special subgroup of $\cW{G}$ is a special subgroup of $\gW{\Gamma}{B}$.
Furthermore, there is
a map $\phi_{\bowtie}\in\Map_\Gamma(B\times B, G/A')$ satisfying
\[
	\phi_{\bowtie}\circ w=\pi(w)\circ \phi_{\bowtie}\qquad (w\in \gW{\Gamma}{B}),
\]
and $\phi\in\Map_\Gamma(B,G/P)$ such that $\phi(x)=\pr_1\circ \phi_{\bowtie}(x,y)$,
$\phi(y)=\pr_2\circ \phi_{\bowtie}(x,y)$.
\end{theorem}

We remark that there is a natural notion of a preorder relation on $\gW{\Gamma}{B}$, generalizing the classical Bruhat order on Coxeter groups,
and one can show that the map $\pi$ considered here is order preserving.
We will not elaborate on this here.

Note that the theorem above could be applied in particular to a lattice $\Gamma$ in a simple Lie group $S$.
Then one deduces some
cases of Margulis superrigidity, e.g. it follows that a lattice in $\SL_n(\bbR)$ cannot
have an unbounded representation in $\SL_m(\bbR)$ if $n>m$.
However, a more efficient approach to superrigidity phenomena with algebraic targets, one that
avoids boundary theory almost completely, is proposed in \cite{BF-csr}.

\subsection*{Commensurator superrigidity}

Finally, let us show how existence and uniqueness of characteristic maps
can be used to prove results analogous to Margulis' commensurator superrigidity \cite{Margulis-book}
(see also \cite{Burger-ACampo}).
Let $S$ be a lcsc group, and $\Gamma<S$ be a lattice. 
Recall that the commensurator of $\Gamma$ in $S$ is the subgroup of $S$ given by
\[
	\Commen_S(\Gamma)=\left\{ s\in S \mid \Gamma\cap \Gamma^s\ \textrm{has\ finite\ index\ in\ }
	\Gamma,\textrm{\ and\ in\ }\Gamma^s\right\},
\]
where $\Gamma^s=\{ g^s=sgs^{-1} \mid g\in\Gamma\}$ denotes conjugation.
\begin{theorem}\label{T:commen}\hfill{}
	Let $S$ be a lcsc group, $\Gamma<S$ a lattice,
	$\Lambda$ a dense subgroup in $S$ such that $\Gamma<\Lambda<\Commen_S(\Gamma)$.
	\begin{itemize}
		\item[{\rm (i)}]
		Let $G\acts M$ be a minimal convergence action and $\rho:\Lambda\to G$ be a continuous homomorphism with a dense image.
		Then $\rho$ extends to a continuous homomorphism $\bar{\rho}:S\to G$.
		\item[{\rm (ii)}]
		Let $\rho:\Lambda\to\Homeo_+(S^1)$ be such that $\Lambda\acts S^1$ acts minimally and strongly proximally.
		Then $\rho$ extends to a continuous homomorphism $\bar{\rho}:S\to \Homeo_+(S^1)$,
		whose image is either discrete or is conjugate to $\PSL_2(\bbR)<\Homeo_+(S^1)$.
		\item[{\rm (iii)}]
		Let $G$ be a connected, simple, center-free, non-compact, real Lie group, and let
		$\rho:\Lambda\to G$ be a homomorphism with Zariski dense image.
		Then $\rho$ extends to a continuous homomorphism $\bar{\rho}:S\to G$.
	\end{itemize}
\end{theorem}
In view of Remark~\ref{R:conv-min}, the assumption of density of $\rho(\Lambda)$ in $G$
and minimality of $G\acts M$ is not restrictive.
Minimality and strong proximality can also be assumed for the circle case,
see the discussion preceding Theorem~\ref{T:char-circle}.
Of course, case (iii) is a special case of the original Margulis commensurator superrigidity,
that was used to give a criterion for arithmeticity of lattices in semi-simple Lie groups.
To this end one needs to consider also algebraic target groups over  $\bbC$, and over
$\bbQ_p$ where $p$ is a prime (cf. \cite{Burger-ACampo}).
We include it here just to emphasize the analogy with the other cases.

\begin{proof}[Sketch of the proof]
	The non-degeneracy assumptions on $\rho(\Lambda)$ are already satisfied by $\rho(\Gamma)$.
	For example in (i), the set $L(\rho(\Gamma))$ being stable under replacing $\Gamma$ by finite index subgroups
	is necessarily $\rho(\Lambda)$-invariant. As $\rho(\Lambda)$ is dense in $G$ and $G\acts M$ is minimal, it
	follows that $L(\rho(\Gamma))=M$.
	Similarly in (ii), one shows that already $\Gamma$ acts minimally and strongly proximally on $S^1$.
	In (iii) $\rho(\Gamma)$ is Zariski dense, because the identity component of its Zariski closure is normalized by $\rho(\Lambda)$
	and hence by all of $G$, but the latter is simple.

	Choose an $S$-boundary $B$, say using Theorem~\ref{T:FP}.
	Then $B$ is a boundary for $\Gamma$, and for any finite index subgroup $\Gamma'<\Gamma$ (Remark~\ref{R:boundaries}).
	Consider the compact $G$-space $Q$ where: in (i) $Q=M$, in (ii) $Q=S^1$, in (iii) $Q=G/P$.
	Then using the above properties of $\rho(\Gamma)$ we obtain a characteristic $\Gamma$-map
	\[
		\phi\in\Map_\Gamma(B,Q)
	\]
	by applying just the first claim in Theorems~\ref{T:char-conv}, ~\ref{T:char-circle},~\ref{T:char-alg} in cases (i), (ii), (iii),
	respectively.
	For any fixed $\lambda\in\Lambda$ the group $\Gamma'=\Gamma\cap \Gamma^{\lambda^{-1}}$ has finite index in $\Gamma$,
	and therefore is also a lattice in $S$.
	The measurable map $\psi:B\to Q$ defined by
	$\psi(x)=\rho(\lambda)^{-1}\phi(\lambda.x)$ is $\Gamma'$-equivariant, because for $g\in\Gamma'$ one has $g^\lambda\in\Gamma$, and so
	\[
		\psi(g.x)=\rho(\lambda)^{-1}\phi(\lambda g.x)=\rho(\lambda)^{-1}\phi(g^\lambda\lambda.x)=\rho(\lambda^{-1} g^\lambda)\phi(\lambda.x)=\rho(g)\psi(x).
	\]
	So both $\phi$ and $\psi$ are in $\Map_{\Gamma'}(B,Q)$, which means that $\phi=\psi$.
	Thus 
	\[
		\phi (\lambda.x)=\rho(\lambda).\phi(x)
	\]
	for a.e. $x\in B$, and this holds for every $\lambda\in\Lambda$.
	Hence the map $\phi:B\to Q$ is $\Lambda$-equivariant.
	This allows to show that $\rho:\Lambda\to G$ extends to a continuous $\bar{\rho}:S\to G$ using, for example,
	the extension Lemma~\ref{L:extension}, once conditions ${\rm (*)}_1$, ${\rm (*)}_2$ have been verified.
	For convergence groups this was done in the proof of Theorem~\ref{T:conv-prod}.
	
	Similar arguments apply to the case (ii) of the circle. One shows using minimality and strong proximality of $\Gamma$-action on $S^1$
	that $\eta=\phi_*\nu$ is a full support continuous measure.
	It easy to see that if $(g_i)_*\eta\to \xi$ for some $g_i\to\infty$ in $\Homeo_+(S^1)$ then $\xi$ has atoms and cannot be $\eta$.
	This proves ${\rm (*)}_1$, while ${\rm (*)}_2$ follows from the fact that $\eta$ has full support.
	
	In case (iii) of $G\acts G/P$, the measure $\eta$ is \emph{proper} meaning that $\eta(V)=0$ for any proper algebraic subspace $V\subset G/P$.
	Then ${\rm (*)}_1$ follows from Furstenberg's lemma about quasi-projective-transformations (\cite{Furst-BD})
	and ${\rm (*)}_2$ is a consequence
	of properness of $\eta$.

	Therefore $\rho:\Lambda\to G$ extends to a continuous homomorphism $\bar{\rho}:S\to G$ in all three cases.
	In the circle case (ii), there is an additional fact: a lcsc group, e.g. $\bar\rho(S)$, with a minimal and strongly
	proximal action on the circle is either discrete or could be conjugated into $\PSL_2(\bbR)$.
\end{proof}


\section{An application to Lyapunov exponents}

In this section we shall apply boundary theory -- Theorem~\ref{T:char-alg} -- to obtain results
about Lyapunov exponents for some matrix valued functions on a class of p.m.p. systems.
In his first proof of superrigidity Margulis used non-vanishing of Lyapunov exponents
for certain matrix valued functions to construct
characteristic maps; our approach (\cite{BF-Lya}) follows a converse direction.

The Multiplicative Ergodic Theorem of Oseledets \cite{Oseled} (see also Kaimanovich \cite{Kaim-MET}, Karlsson-Margulis \cite{KM})
describes the asymptotic behavior of products of matrix valued functions along orbits of a p.m.p. system.
More precisely, let $(X,m,T)$ be an invertible, ergodic, p.m.p. system, and $F:X\to \SL_d(\bbR)$ a measurable map with
\begin{equation}\label{e:L1}
	\int_X \log\|F(x)\|\,dm(x)<+\infty.
\end{equation}
Multiplying $F$ along $T$-orbits one obtains a measurable cocycle $\bbZ\times X\to \SL_d(\bbR)$
\begin{equation}\label{e:cocycle}
	\begin{split}
	F_n(x)=\left\{\begin{array}{lll}
	F(T^{n-1}x)\cdots F(Tx)F(x) & \textrm{if} & n\ge 1,\\
	I & \textrm{if} & n=0,\\
	F(T^nx)^{-1}\cdots F(T^{-1}x)^{-1} & \textrm{if} & n<0.
	\end{array}\right.
	\end{split}
\end{equation}
The cocycle equation being $F_{k+n}(x)=F_k(T^n x)F_n(x)$ for $k,n\in\bbZ$.
The Multiplicative Ergodic Theorem asserts that there exist: a partition $d=d_1+\dots+d_s$, constants
\[
	\gamma_1>\gamma_2>\dots>\gamma_s\qquad \textrm{with}\qquad d_1\gamma_1+\dots+d_s\gamma_s=0
\]
and a measurable equivariant\footnote{satisfying $m$-a.e. $L_j(T^n x)=F_n(x) L_j(x)$ for $1\le j\le s$.} splitting into vector subspaces
\[
	\bbR^d =L_1(x)\oplus L_2(x)\oplus\dots\oplus L_s(x),\qquad \dim L_j(x)=d_j,
\]
so that for $m$-a.e. $x\in X$ for every $v\in L_j(x)\setminus\{0\}$ one has
\[
	\lim_{n\to\infty} \frac{1}{n}\log\|F_n(x)v\|=\gamma_j,\qquad \lim_{n\to\infty} \frac{1}{n}\log\|F_{-n}(x)v\|=-\gamma_j.
\]
Rewriting $\gamma_1>\gamma_2>\dots>\gamma_s$ with their multiplicities we get the \emph{Lyapunov exponents}
\[
	\lambda_1\ge\lambda_2\ge \dots\ge \lambda_d,\qquad \lambda_1+\dots+\lambda_d=0,
\]
that can be recorded as $\Lambda={\rm diag}(\lambda_1,\dots,\lambda_d)$; we refer to
this element $\Lambda\in \frak{sl}_d(\bbR)$
as the \emph{Lyapunov spectrum} of $F$ on $(X,m,T)$.
The Multiplicative Ergodic Theorem can be restated as the assertion that a.e.
sequence $\{F_n(x)\}_{n\in\bbZ}$ follows with a sub-linear deviation
the sequence $\{U_x \exp(n\Lambda) V^{-1}_x\}_{n\in\bbZ}$ for some $U_x,V_x\in {\rm O}_d(\bbR)$.
More generally, given a simple real Lie group $G$ and an integrable\footnote{
(\ref{e:L1}) holds for some/any embedding $G<\SL_d(\bbR)$}
measurable map $F:X\to G$, we define the \emph{Lyapunov spectrum} to be the element in the positive Weyl chamber
\[
	\Lambda \in \frak{a}_+
\]
of the Cartan subalgebra $\frak{a}$ of $\frak{g}={\rm Lie}(G)$, so that $u_x\exp(n\Lambda)v_x^{-1}$
represent the asymptotic behavior of $m$-a.e. sequences $F_n(x)\in G$, $n\in\bbZ$.

The spectrum $\Lambda$ of an integrable $F:X\to G$ over an ergodic invertible system $(X,m,T)$
is \emph{non-degenerate} if $\Lambda\ne 0$,
and is called \emph{simple} if $\Lambda$ is a regular element in $\frak{a}_+$.
In the basic case $G=\SL_d(\bbR)$, non-degeneracy of the spectrum corresponds to $\lambda_1>0$,
and simplicity to the strict inequalities
\[
	\lambda_1>\lambda_2>\dots>\lambda_d
\]
in which case $s=d$ and $d_j=\dim L_j(x)=1$ for all $1\le j\le d$.

In general, there is no explicit formula for $\Lambda$ (or even $\lambda_1$) in terms of $F:X\to G$ on $(X,m,T)$,
and the dependence of the Lyapunov exponents on $F$ and $(X,m,T)$ is mostly mysterious.
The best studied situation is that of Random Walks, where $(X,m)$ is the invertible Bernoulli shift
$(G^\bbZ,\mu^\bbZ)$ with $(Tx)_i=x_{i+1}$ and $F(x)=x_1$.
From the fundamental work of Furstenberg \cite{Furst-Poisson}, Guivarc'h-Raugi \cite{GR, GR2}, and Gol'dsheid-Margulis \cite{GM},
it is known that if $\supp(\mu)$ generates a Zariski dense subgroup in $G$, then the Lyapunov spectrum is simple.
More recently, Avila and Viana \cite{AV, AV2, AV3} gave sufficient conditions for simplicity of the Lyapunov
for certain classes of systems that allowed them to prove simplicity of the Lyapunov spectrum of Kontsevich-Zorich cocycle.
Here we shall describe an approach (\cite{BF-Lya}) that allows to prove simplicity of the Lyapunov spectrum
using boundary theory.

\subsection*{Simplicity of the Lyapunov spectrum}

Let $(X,m,T)$ be an ergodic, invertible, p.m.p. system and $\Gamma$ be some auxiliary group,
that we assume to be countable discrete for clarity of presentation, and let
\[
	f:X\to \Gamma
\]
be a measurable map. It generates a measurable cocycle $\bbZ\times X\to\Gamma$, denoted $f_n(x)$, similarly to (\ref{e:cocycle}).
Let $\wt{m}$ denote the (infinite) measure on the space $\Gamma^\bbZ$ obtained by pushing forward
the product $m\times c_\Gamma$ of $m$ with the counting measure $c_\Gamma$ on $\Gamma$ by the map
\[
	 X\times \Gamma\ \overto{}\ \Gamma^\bbZ,\qquad (x,g)\mapsto (f_n(x)g^{-1})_{n\in\bbZ}.
\]
This measure describes the distribution of paths $(g_i)$ of a stochastic walk with (not necessarily independent)
increments $f(T^nx)=g_{i+1}g_i^{-1}$ that starts from an arbitrary initial value $g_0$.
Let us write $\wt{X}$ for the space $\Gamma^\bbZ$ with the measure $\wt{m}$ or another measure in its measure class.
The measure $\wt{m}$ (and its class) are preserved by the commuting actions of $\bbZ$ and $\Gamma$:
\[
	n:(g_i)\mapsto (g_{i+n}),\qquad g:(g_i)\mapsto (g_ig^{-1})\qquad (n\in\bbZ,\ g\in\Gamma).
\]
%
Consider the future tail equivalence relation $\sim_+$ on $\wt{X}$
defined by $(g_i)\sim_+(g'_i)$ if for some $k\in\bbZ$ one has $g_{i+k}=g'_i$ for all $i\ge i_0$.
Let $B_+=\wt{X}\ec\sim_+$  denote the space of $\sim_+$-ergodic components.
To make this more precise, one may replace $\wt{m}$ by an equivalent probability measure $\wt{m}_1$,
push down $\wt{m}_1$ by the projection $\Gamma^\bbZ\to \Gamma^\bbN$, and take
the ergodic components for the semi-group $\bbN$ acting by the shift.
Then $B_+$ is a measurable $\Gamma$-space which is a quotient of $\wt{X}$.
Similarly, one defines the past tail equivalence relation $\sim_-$ and the corresponding $\Gamma$-quotient $\wt{X}\to B_-$.
We shall say that the quotients $p_-:\wt{X}\to B_-$ and $p_+:\wt{X}\to B_+$ are \emph{weakly independent}, denoted
$B_-\perp B_+$, if
\begin{equation}\label{e:wind}
	(p_-\times p_+)[\wt{m}]\,=\, p_-[\wt{m}]\times p_+[\wt{m}],
\end{equation}
where $[\wt{m}]$ denote the measure class of $\wt{m}$, and the equality is of measure classes.
%

\begin{example}\label{E:RW}
	Let $\mu$ be a generating probability measure on a (countable) group $\Gamma$,
	$(X,m,T)$ be the Bernoulli system $(\Gamma^\bbZ,\mu^\bbZ)$ with the shift $T:(x_i)\mapsto(x_{i+1})$,
	and $f:X\to\Gamma$ given by $f(x)=x_1$.
	Then $\wt{X}$ is the space of paths for random walks and $B_+$ and $B_-$ are the Furstenberg-Poisson boundaries
	for $\mu$ and $\check\mu$ respectively. They are weakly independent $B_+\perp B_-$.
	Note that the assumption that $\mu$ is generating is essential here, for if $\mu$ is supported
	on a proper subgroup $\Gamma_0<\Gamma$ then the non-trivial $\Gamma$-space $\Gamma/\Gamma_0$
	is a common quotient of $\wt{X}$, $B_+$, $B_-$, and $B_+\not\perp B_-$.
\end{example}

\begin{example}\label{E:geo}
	Let $M$ be a closed Riemannian manifold of negative curvature, $\Gamma=\pi_1(M)$ the fundamental group,
	$X=SM$ unit tangent bundle, $T$ the time one geodesic flow on $X$,
	and $m$ be Lebesgue-Liouville measure, or Bowen-Margulis measure, or any other Gibbs measure.
	We define a cocycle $f_n:X\to \Gamma$, $n\in\bbZ$, by
	\[
		\wt{T}^n(\sigma(x))=f_n(x).\sigma(T^nx)
	\]
	where $\wt{T}^t$ is the geodesic flow on the unit tangent bundle $S\wt{M}$ of the universal cover $\wt{M}$,
	and $\sigma:SM\to S\wt{M}$ is the section of the covering map $\pi:S\wt{M}\to SM$ corresponding to a measurable
	choice of a fundamental domain, say a Dirichlet domain.
	The $B_-$ and $B_+$ are then realized on the geometric boundary $\partial\wt{M}$
	and the measure classes represent those of stable/unstable foliations.
	One has weak independence $B_-\perp B_+$ as a consequence of the local product structure of the conditional measures on
	stable/unstable leaves, and the mixing condition.
\end{example}

\begin{theorem}\label{T:simple-Lya}\hfill{}\\
	Let $(X,m,T)$ be an invertible ergodic p.m.p. system, $f:X\to \Gamma$ a measurable map, so that
	$B_-\perp B_+$ in the above sense. Then
	\begin{itemize}
		\item[{\rm (i)}]
		Let $G$ be a connected, non-compact, center free, simple, real Lie group and $\rho:\Gamma\to G$
		a representation with Zariski dense image.
		Then the map
		\[
			F:X\ \overto{f}\ \Gamma\ \overto{\rho}\ G
		\] 
		has simple Lyapunov spectrum over $(X,m,T)$, provided it is integrable.
		\item[{\rm (ii)}]
		Let $\Gamma\acts (Z,\zeta)$ be an ergodic p.m.p. action, and $\rho:\Gamma\times Z\to G$ a
		Zariski dense cocycle into $G$ as above.
		Then the skew-product
		\begin{equation}\label{e:skewproduct}
			(X\times Z,m\times\zeta,T_f),\qquad T_f:(x,z)\mapsto (Tx,f(x).z)
		\end{equation}
		is ergodic, and the map $F:X\times Z\to G$,  $F(x,z)=\rho(f(x),z)$,
		has simple Lyapunov spectrum provided it is integrable.
	\end{itemize}
	Non-degeneracy of the Lyapunov spectrum ($\lambda_1>0$) remains valid if Zariski density condition on $\rho$ is replaced
	by the weaker condition that the algebraic hull of $\rho$ is non-amenable.
\end{theorem}
Let us note some consequences of this result.
In the random walk setting (Example \ref{E:RW}) we recover the results of Guivarc'h-Raugi, Gol'dsheid-Margulis on simplicity of the Lyapunov spectrum
for Zariski dense random walk on a simple Lie group $G$ by applying part (i) of the theorem to $X=G^\bbZ$, $m=\mu^\bbZ$ with the shift.
The addendum about non-degeneracy of the spectrum is precisely Furstenberg's condition for $\lambda_1>0$.
Part (ii) gives already a new result:
\begin{corollary}
	Let $\Gamma$ be a (countable) group, $\Gamma\acts (Z,\zeta)$ an ergodic p.m.p. action, $\rho:\Gamma\times Z\to G$ a Zariski dense cocycle.
	Let $\mu$ be a generating probability measure on $\Gamma$ with $\log\|\rho(g,z)\|\in L^1(\mu\times\zeta)$, let $X=\Gamma^\bbZ$, $m=\mu^\bbZ$,
	$T:(x_i)\mapsto(x_{i+1})$, and $\ol{T}:(x,z)\mapsto (Tx,x_1.z)$. Then the cocycle $F_n:X\times Z\to G$ given by
	\[
		F_n(x,z)=\rho(x_n\cdots x_1,z)
	\]
	has a simple Lyapunov spectrum.
	If $\rho$ is only assumed to have non-amenable algebraic hull, then the spectrum is non-degenerate.
\end{corollary}
The result about non-degeneracy of the Lyapunov spectrum in this setting is due to Ledrappier \cite{Led}.

\begin{corollary}
	Let $M$ be a compact negatively curved manifold, $T^t$ the geodesic flow on the unit tangent bundle $X=SM$ to $M$, $m$ a Gibbs measure on $X$,
	and $f_n:X\to \Gamma=\pi_1(M)$ a cocycle as in Example~\ref{E:geo}. Then
	\begin{itemize}
		\item[{\rm (i)}] Given a Zariski dense representation $\rho:\Gamma\to G$ in a simple Lie group, the Lyapunov spectrum of
		$F=\rho\circ f$ is simple.
		\item[{\rm (ii)}] Given any ergodic p.m.p. action $\Gamma\acts (Z,\zeta)$ the skew-product $X\times_f Z$ is ergodic
		and if $\rho:\Gamma\times G$ is a Zariski dense cocycle with $\log\|\rho(g,-)\|\in L^1(Z)$, $g\in\Gamma$, then
		the Lyapunov spectrum of $F_n(x,z)=\rho(f_n(x),z)$ is simple.
	\end{itemize}
\end{corollary}
Part (ii) for the case where $M$ is a constant curvature surface, can be restated as asserting that for an ergodic p.m.p.
action $\SL_2(\bbR)\acts(X,m)$ and Zariski dense integrable cocycle $\rho:\PSL_2(\bbR)\times X\to G$
the restriction to the diagonal subgroup $F_n(x)=\rho(g^n,x)$ for some hyperbolic $g\in\PSL_2(\bbR)$, has simple Lyapunov spectrum.
This was recently obtained by Eskin-Matheus \cite{Eskin+Matheus}.

\subsection*{Outline of the proof} 

The main observation is that the setting of $f:X\to\Gamma$ and condition $B_-\perp B_+$,
described above, allow one to use boundary theory.
\begin{theorem}\label{T:wind-boundaries}\hfill{}\\
	Let $(X,m,T)$ and $f:X\to \Gamma$ be as above, and assume that $B_-\perp B_+$.\\
	Then $\wt{X}\ec\bbZ$ is isometrically ergodic, projections $\wt{X}\ec\bbZ\to B_\pm$ are relatively isometrically ergodic,
	and $(B_-, B_+)$ are boundary pair for $\Gamma$.
\end{theorem}
We shall not describe the proof of this result here, but remark that amenability of $B_\pm$ follows from amenability of $\bbN$
(as in Zimmer's \cite{Zimmer-amen}), and other statements reduce to
relative isometric ergodicity of the maps $\wt{X}\ec\bbZ\to B_\pm$.
The proof of this key property is motivated by Kaimanovich \cite{Kaim-DE}.

\begin{observation}\label{O:equiv-inv}
	Let $V$ be a Borel $\Gamma$-space, and $\Map_f(X,V)$ denote the space
	of all $f$-\emph{equivariant maps}, i.e. measurable maps $\phi:X\to V$ satisfying $m$-a.e. $\phi(Tx)=f(x).\phi(x)$.
	Then there exists a natural bijection between $f$-equivariant maps and $\bbZ$-invariant $\Gamma$-equivariant maps $\wt{X}\to V$,
	which gives a bijection
	\[
		\Map_f(X,V)\ \cong\ \Map_\Gamma(\wt{X}\ec\bbZ,V).
	\]
\end{observation}
This observation gives the following fact, that was included in the statement of Theorem~\ref{T:simple-Lya}.(ii).
\begin{corollary}[of Theorem~\ref{T:wind-boundaries}]\label{C:Moore}\hfill{}\\
	Let $(X,m,T)$ and $f:X\to \Gamma$ be such that $B_-\perp B_+$.
	Then for any ergodic p.m.p. action $\Gamma\acts (Z,\zeta)$
	the skew-product (\ref{e:skewproduct}) is ergodic.
\end{corollary}
Note that for the random walk setting (Example~\ref{E:RW}) this can be deduced from Kakutani's random ergodic theorem,
and for the geodesic flow setting (Example~\ref{E:geo}) with $M$ being \emph{locally symmetric}, it follows from Moore's ergodicity.
However it is new for geodesic flow on general negatively curved manifolds, and potentially in other situations.
\begin{proof}
	The claim is that $T_f$-invariant functions $F\in L^2(X\times Z,m\times\zeta)$ are a.e. constant.
	Such an $F$ can be viewed
	as a measurable $f$-equivariant map $X\to L^2(Z,\zeta)$, $x\mapsto F(x,-)$.
	By \ref{O:equiv-inv} it corresponds to a $\Gamma$-map $\Phi:\wt{X}\ec\bbZ\to L^2(Z,\zeta)$.
	Since $\wt{X}\ec\bbZ$ is isometrically ergodic (Theorem~\ref{T:wind-boundaries}),
	$\Phi$ is constant $\phi_0\in L^2(Z,\zeta)^\Gamma$.
	As $\Gamma\acts (Z,\zeta)$ is ergodic, $\phi_0$ is $\zeta$-a.e. a constant $c_0$, and $F$ is $m\times\zeta$-a.e. constant $F(x,z)=c_0$.
\end{proof}

\bigskip

Let us outline the proof of Theorem~\ref{T:simple-Lya}.
We focus on part (i) that refers to the simplicity of the Lyapunov spectrum of $F=\rho\circ f:X\to G$ where $\rho:\Gamma\to G$ is a Zarsiki dense representation.
The proof of part (ii) that refers to cocycles follows the same outline.

By Theorem~\ref{T:wind-boundaries}, the pair $(B_-, B_+)$ constructed from $(X,m,T)$ and $f:X\to\Gamma$ is a boundary pair for $\Gamma$.
Therefore from Theorem~\ref{T:char-alg} there exist $\Gamma$-maps
\[
	\phi_-:B_-\ \overto{}\ G/P,\qquad \phi_+:B_-\ \overto{}\ G/P,\qquad \phi_{\bowtie}:B_-\times B_+\ \overto{}\ G/A'
\]
so that $\phi_-=\pr_1\circ \phi_{\bowtie}$ and $\phi_-=\pr_2\circ \phi_{\bowtie}$, where
$G/A'$ is viewed as a subset
\[
	G/A'\subset G/P\times G/P.
\]
For $n\in\bbZ$ denote by $\calF_{\ge n}=\sigma(f\circ T^n, f\circ T^{n+1},\dots)$ the $\sigma$-algebra generated by
the maps $f\circ T^k:X\to\Gamma$, $k\ge n$. Similarly define $\calF_{<n}=\sigma(f\circ T^{n-1}, f\circ T^{n-2},\dots)$.
Then $\calF_{\ge n}\subset \calF_{\ge n-1}$ and $\calF_{< n}\supset \calF_{<n-1}$.

\begin{proposition}\label{P:harmonic}\hfill{}\\
	There exists a map $\nu_-:X\to \Prob(G/P)$ with the following properties:
	\begin{itemize}
		\item[{\rm (i)}]
		The map $x\mapsto \nu_-(x)$ is $\calF_{\ge 0}$-measurable and satisfies
		\[
			\nu_-(x) =\bbE\left( F(T^{-1}x)_*\nu_-(T^{-1}x) \mid \calF_{\ge 0}\right).
		\]
		\item[{\rm (ii)}]
		For $m$-a.e. $x\in X$ there is weak-* convergence to Dirac measure
		\[
			\delta_{\psi_-(x)}=\lim_{n\to\infty} F(T^{-1}x)F(T^{-2}x)\cdots F(T^{-n}x)_*\nu_-(T^{-n}x),
		\]
		where $\psi_-$ is an $F$-equivariant map $X\to G/P$.
		\item[{\rm (iii)}]
		For $m$-a.e. $x\in X$ the measure $\nu_-(x)$ is proper, i.e. gives zero mass to proper algebraic subspaces $W\subsetneq G/P$.
	\end{itemize}
	There is a $\calF_{<0}$-measurable map $\nu^+:X\to\Prob(G/P)$ and $\psi_+\in\Map_F(X,G/P)$
	with similar properties with respect to $T^{-1}$. Moreover, there exists
	\[
		\psi_{\bowtie}\in \Map_F(X,G/A'),\qquad \textrm{so\ that}\qquad
		\psi_-=\pr_1\circ \psi_{\bowtie},\qquad \psi_+=\pr_2\circ \psi_{\bowtie}
	\]
	where $\pr_i:G/A'\to G/P$ are the projections.
\end{proposition}
\begin{proof}[Sketch of the proof]
	The map $\psi_-\in\Map_F(X,G/P)$ is defined by applying the correspondence from \ref{O:equiv-inv} to the pull-back
	of $\phi_-\in \Map_\Gamma(B_-,G/P)$ via the quotient $\wt{X}\ec\bbZ\to B_-$.
	Define $\nu_-$ to be the conditional expectation (average) of the Dirac measures $\delta_{\psi_-(x)}$
	\[
		\nu_-(x) =\ \bbE \left(\delta_{\psi_-(x)} \mid \calF_{\ge 0}\right).
	\]
	Property (i) then follows from this definition, and (ii) follows by applying Martingale Convergence Theorem.
	
	We shall not give here the proof of property (iii), but point out that it uses the $B_-\perp B_+$ assumption
	as a well as Zariski density of $\rho$.
\end{proof}

The following well known lemma allows one to prove quantitative results (linear growth of ergodic sums)
from qualitative information (consistent growth of ergodic sums).
\begin{lemma}\label{L:Kesten}
	Let $(X,m,T)$ be an ergodic p.m.p. system, and $h\in L^1(X,m)$ such that $h(x)+h(Tx)+\dots+h(T^nx)\to+\infty$
	for $m$-a.e. $x\in X$. Then $\int h\,dm>0$.
\end{lemma}

Contraction of measures on $G/P$ can indicate growth.
\begin{lemma}
	Let $Q\subset \Prob(G/P)$ be a compact set of proper measures, $\{\nu_n\}$ a sequence in $Q$,
	and let $\{a_n\}$ be a sequence in the Cartan subalgebra $\frak{a}$ of $\frak{g}={\rm Lie}(G)$, so that
	\[
		\exp(a_n)_*\nu_n\ \overto{}\ \delta_{eP}.
	\]
	Then for any positive root, $\chi:\frak{a}\to\bbR$ one has $\chi(a_n)\to\infty$.
\end{lemma}
Combining these two Lemmas, one may deduce simplicity of the spectrum in the following very special situation:
assume that
\begin{itemize}
	\item an integrable $F:X\to G$ takes values in the Cartan subgroup $A=\exp(\frak{a})$, so we can write $F(x)=\exp(a(x))$
	for an appropriate function $a:X\to \frak{a}$,
	\item
	some map $\nu:X\to \Prob(G/P)$, taking values in a compact $Q\subset\Prob(G/P)$ consisting of proper measures,
	satisfies weak-* convergence
	\[
		F(T^{-1}x)F(T^{-2}x)\cdots F(T^{-n}x)_*\nu_-(T^{-n}x)\ \overto{}\ \delta_{eP}.
	\]
\end{itemize}
Then one has $\Lambda=\int a\,dm$, and since $\chi(\Lambda)=\int\chi(a(x))\,dm(x)>0$ for every positive root $\chi$, the spectrum $\Lambda$ is simple.

Returning to the general case described in Proposition~\ref{P:harmonic}, one can use
the maps $\psi_{\bowtie}:X\to G/A'$ and $\psi_-:X\to G/P$ to find a measurable $c:X\to G$ so that
\[
	c(Tx) F(x) c(x)^{-1}\in A=\exp(\frak{a}),\qquad\textrm{and}\qquad c(x)\psi_-(x)=eP,
\]
while all $\nu(x)=c(x)\nu_-(x)$ still remain proper measures.
To arrive at the special situation described above we need to control integrability of the $A$-valued $c(Tx) F(x) c(x)^{-1}$
and to ensure uniform properness for $\nu(x)$. This can be achieved by passing to an \emph{induced system} (in the sense of Kakutani)
as follows.
There exist compact sets $C\subset G$ and $Q\subset \Prob(G/P)$ where $Q$ consists of proper measures only,
so that the set
\[
	X^*=\{ x\in X \mid c(x)\in C,\ \nu(x)\in Q\}
\]
has $m(X^*)>0$. Let $m^*$ be the normalized restriction $m^*=m(X^*)^{-1}\cdot m|_{X^*}$, denote the first return time to $X^*$ by
$n(x)=\inf\{ n\ge 1 \mid T^n x\in X^*\}$, and define
\[
	T^*x=T^{n(x)}x,\qquad F^*(x)=c(T^{n(x)}x)F_{n(x)}(x)c(x)^{-1}.
\]
From the ergodic theorem $\int n(x)\,dm^*(x)=m(X^*)^{-1}$, and it follows that the Lyapunov spectra, $\Lambda$ of $F$ on $(X,m,T)$
and $\Lambda^*$ of $F^*$ on $(X^*,m^*,T^*)$, are positively proportional
\[
	\Lambda^*=\frac{1}{m(X^*)}\cdot \Lambda \in \frak{a}^+.
\]
But $F^*$ on $(X^*,m^*,T^*)$ satisfies the condition of the special case above, hence $\Lambda^*$ is simple, and therefore
so is the original $\Lambda\in\frak{a}^+$.

Finally the addendum about non-degeneracy of the Lyapunov spectrum when $\rho(\Gamma)$  is just assumed to be non-amenable,
follows from the simplicity criterion by considering the Levi decomposition of the Zariski closure of $\rho(\Gamma)$.
This completes the outline of the proof of Theorem~\ref{T:simple-Lya}.

\frenchspacing

%


\end{document}